\documentclass[11pt]{article}
\usepackage{amsmath,amssymb,frcursive,mathrsfs}
\usepackage{graphicx,tikz,pgfplots,pgfplotstable}

\usepackage{fullpage}
\usepackage[english]{babel}
\usepackage[utf8]{inputenc}
\usepackage{amsmath,amsthm,amssymb,frcursive,mathrsfs,mathtools,stmaryrd}
\usepackage{color}
\usepackage{hyperref}             
\hypersetup{linktocpage}
\usepackage{versions}

\pgfplotsset{every axis/.append style={
tick label style={font=\scriptsize}  
}}

\newtheorem{example}{Example}[section]
\newtheorem{definition}{Definition}[section]
\newtheorem{theorem}{Theorem}
\newtheorem{proposition}{Proposition}[section]
\newtheorem{corollary}[proposition]{Corollary}
\newtheorem{lemma}[proposition]{Lemma}
\newtheorem{remark}{Remark}[section]
\numberwithin{figure}{section}

\DeclareMathOperator*{\argmin}{arg\, min}

\DeclareMathOperator{\Lip}{\Lip}

\newcommand{\cB}{{\cal B}}

\newcommand{\cE}{{\cal E}}

\newcommand{\cL}{\mathscr{L}}
\newcommand{\cM}{{\cal M}}
\newcommand{\cO}{{\cal O}}

\newcommand{\cS}{{\cal S}}
\newcommand{\cT}{{\cal T}}

\newcommand{\N}{\mathbb{N}}
\newcommand{\R}{\mathbb{R}}
\newcommand{\PP}{{\mathbb P}}

\newcommand{\tin}{\text{ in }}
\newcommand{\ton}{\text{ on }}

\newcommand{\ra}{\rightarrow}
\newcommand{\bs}{\backslash}

\newcommand{\wt}{\widetilde}

\newcommand{\e}{{\varepsilon}}

\newcommand{\g}[1]{\boldsymbol{#1}}
\newcommand{\norm}[2]{\left\|{#1}\right\|_{#2}}
\newcommand{\inner}[2]{\left<{#1}\right>_{#2}}

\newcommand{\st}{\text{s.t.}}

\makeindex
\title{Stability for finite element discretization of some elliptic inverse parameter problems from internal data - application to elastography}

\author{Elie Bretin\thanks{Institut Camille Jordan, INSA de Lyon \& UCBL, 69003 Lyon, France.}\and Pierre Millien\thanks{Institut Langevin, CNRS UMR 7587, ESPCI Paris, PSL Research University, 1 Rue Jussieu, 75005 Paris, France.} \and Laurent Seppecher\thanks{Institut Camille Jordan, Ecole Centrale de Lyon \& UCBL, Lyon, F-69003, France.}}

\begin{document}
\maketitle

\begin{abstract}
In this article, we provide stability estimates for the finite element discretization of a class of inverse parameter problems 
of the form $-\nabla\cdot(\mu S) =  \g f$ in a domain $\Omega$ of $\R^d$. Here $\mu$ is the unknown parameter 
to recover, the matrix valued function $S$ and the vector valued distribution $\g f$ are known. As uniqueness
is not guaranteed in general for this problem, we prove a Lipschitz-type stability estimate in an
hyperplane of $L^2(\Omega)$. This stability is obtained through an adaptation of the so-called
discrete \emph{inf-sup} constant or LBB constant to a large class of first-order differential operators. 
We then provide a simple and original discretization based on hexagonal finite element that satisfies the discrete
stability condition and shows corresponding numerical reconstructions. The obtained algebraic inversion method is efficient
as it does not require any iterative solving of the forward problem and is very general as it does not require any 
smoothness hypothesis for the data nor any additional information at the boundary.
\end{abstract}



\def\keywords2{\vspace{.5em}{\textbf{  Mathematics Subject Classification
(MSC2000).}~\,\relax}}
\def\endkeywords2{\par}
\keywords2{65J22, 65N21, 35R30, 65M60}

\def\keywords{\vspace{.5em}{\textbf{ Keywords.}~\,\relax}}
\def\endkeywords{\par}
\keywords{Inverse problems, Reverse Weak Formulation, Inf-Sup constant, Linear Elastography, Finite Element Method}

\section{Introduction}


This work deals with inverse problems of the form
\begin{equation}\label{eq:1}
-\nabla\cdot(\mu S) =  \g f\quad\tin\Omega,
\end{equation}
where $\Omega$ is a smooth bounded domain of $\R^d$, $d\geq 2$ and where $\mu \in L^\infty(\Omega)$ is the unknown parameter map. In this problem, $S\in L^\infty(\Omega,\R^{d\times d})$ and $\g f\in H^{-1}(\Omega,\R^d)$ are given from some measurements and may contain noise. If one defines the first order differential operator 
\begin{equation}\label{eq:T1}
\begin{aligned}
T:L^\infty(\Omega)\subset L^2(\Omega)&\to H^{-1}(\Omega,\R^d)\\
\mu &\mapsto -\nabla\cdot(\mu S),
\end{aligned}
\end{equation}
the inverse parameter problem that we aim to solve can be expressed as 
\begin{equation}\label{eq:pbconti}
\text{Find }\mu\in L^\infty(\Omega)\quad\text{s.t.}\quad T\mu = \g f.
\end{equation}
As the right-hand side $\g f$ belongs to $H^{-1}(\Omega,\R^d)$ the meaning of this problem as to be understood through its corresponding Reverse Weak  Formulation (RWF):
\begin{equation}\label{eq:RWF}
\text{Find }\mu\in L^\infty(\Omega)\quad\text{s.t.}\quad\inner{T\mu,\g v}{H^{-1},H^1_0} = \inner{T\mu,\g v}{H^{-1},H^1_0},\quad\forall \g v\in H^1_0(\Omega,\R^d).
\end{equation}
In this inverse problem, we do not assume the knowledge of any information on $\mu$ at the boundary nor additional smoothness hypothesis. Note that the case $\g f=\g 0$ can be considered and corresponds to the determination of the null space the operator $T$.

The goal of the present paper is to investigate the stability properties of the discretized version of the problem \eqref{eq:RWF} and to provide error estimates based on the properties of the discretization spaces and on the discretized 
approximation of the operator $T$. These estimates do not require any regularization technique. More precisely, given a finite dimensional operator $T_h:M_h\ra V_h'$ and $\g f_h\in V_h'$ where $M_h$ and $V_h$ are finite dimensional subspaces that approach $M:= L^2(\Omega)$ and $V :=  H^{1}_0(\Omega,\R^d)$ 
 respectively,  we seek conditions on $M_h$, $V_h$ and $T_h$ for the $L^2$-stability of the following discretized problem:
 \begin{equation}\label{eq:Th}
\text{Find }\mu_h\in M_h\quad\st\quad T_h\mu_h = \g f_h.
\end{equation}
We also give conditions that guarantee the convergence of $\mu_h$ to $\mu$ for the $L^2$-norm. In most cases, the stability only occurs in an hyperplane of $L^2(\Omega)$. This leads to a remaining scalar uncertainty that can be resolved using a single additional scalar information on $\mu$.

%
%
%
%

The originality of this work lies here on the Reverse Weak Formulation \eqref{eq:RWF} that exhibits the unknown parameter $\mu$ as the solution of a weak linear differential problem in the domain $\Omega$ without boundary condition. Hence the uniqueness is not guaranteed at first look and the stability has to be considered with respect to some possible errors on both $\g f$ and $T$. As we will see, the error term $T_h-T$ is not controlled in $\cL\left(L^2(\Omega),H^{-1}(\Omega,\R^d)\right)$ (definition in Section \ref{sec:discretisation}) in general but only for a weaker norm (see Subsection \ref{sub:interp}). This creates difficulties that are not covered by the classic literature on the theory of perturbations of linear operators.


\subsection{Scientific context and motivations}

Elastography is an imaging modality that aims at reconstructing the mechanical properties of biological tissues. The local values of  the elastic parameters can be used as a discriminatory criterion for differentiating healthy tissues from diseased tissues \cite{sarvazyan1995biophysical}. While numerous modalities of elastography exist (see the for example \cite{gennisson2013ultrasound,parker2010imaging,doyley2012model,brusseau20082}), the most common procedure is to use an auxiliary imaging method (such as ultrasound imaging, magnetic resonance imaging, optical coherence tomography \ldots) to measure the displacement field $\g u$ in a medium when a mechanical perturbation is applied. See \cite{sherina2021challenges} and inside references for recent advances on this point. The inverse problem can be formulated as recovering the shear modulus $\mu$ in the linear elastic equation\begin{equation}\label{eq:elastic}
-\nabla\cdot(2\mu\cE(\g u))-\nabla(\lambda\nabla\cdot \g u) =  \g f\quad\tin\Omega,
\end{equation}
where $\g u$ and $\g f$ are given in $\Omega$ and $\lambda$ ca be assumed known in $\Omega$. The term $\cE(\g u)$ denotes the strain matrix which is the symmetric part of the gradient of $\g u$. 
The stability of this inverse problem has been extensively studied under various regularity assumptions for the coefficients to be reconstructed \cite{ammari2015stability,bal2015reconstruction,widlak2015stability,hubmer2018lame}. Recently, in \cite{ammari2021direct} the authors introduced a new inversion method based on a finite element discretization of equation \eqref{eq:1} where $S:=2 \cE(\g u)$. 
A study of the linear operator $T$ defined by \eqref{eq:T1} or by the equivalent weak formulation 
\begin{equation}\label{eq:T1VF}
\begin{aligned}
\inner{T\mu,\g v}{H^{-1},H^1_0}:=\int_\Omega\mu S:\nabla \g v,\quad \forall \g v\in H^1_0(\Omega,\R^{d\times d})
\end{aligned}
\end{equation}
showed that, under a piecewise smoothness hypothesis on $S$ and under an assumption of the form $|\det(S)|\geq c>0\ a.e.$ in $\Omega$, the operator $T$ has a null space of dimension one at most and is a closed range operator. This ensures the theoretical stability of the reconstruction in the orthogonal complement of the null space. However, depending on the choice of discretization spaces, the discretized version of $T$ may not satisfy the same properties and numerical instability may be observed. For instance, in \cite{ammari2021direct} the authors approach \eqref{eq:T1VF} using the classical pair $(\PP^0,\PP^1)$ of finite element spaces. As it could have been expected, they faced a numerical instability that was successfully overcome by using a $TV$-penalization technique. 

\begin{remark} The classic elliptic theory says that the strain matrix belongs to $L^2(\Omega,\R^{d\times d})$. Here, we add the hypothesis $S\in L^\infty(\Omega)(\Omega,\R^{d\times d})$ in order to control the error on $\mu$ in the Hilbert space $L^2(\Omega)$. This smoothness hypothesis is not very restrictive as it is known that the strain is bounded as soon as the elastic parameters are piecewise smooth with smooth surfaces of discontinuity. 
\end{remark}


Let us point out here that inverse problems of the form \eqref{eq:1} may arise from various other physical situations. Note first that the reconstruction of the Young's modulus $E$ when the Poisson's ratio $\nu$ is known is very similar to the problem defined in \eqref{eq:elastic}. In this case the governing linear elastic equation reads $-\nabla\cdot \left(E\, \Sigma \right) =  \g f$ where $\Sigma := a_\nu\cE(\g u)+ b_\nu(\nabla\cdot \g u) I$ and $a_\nu:=1/(1+\nu)$ and $b_\nu:=\nu/((1+\nu)(1-2\nu)$ in dimension $d=3$. A second example is the electrical impedance imaging with internal data, where the goal is to recover the conductivity $\sigma$ in the scalar elliptic equation $-\nabla\cdot(\sigma \nabla u) = 0$. If one can measure two potential fields $u_1$ and $u_2$ solutions of the previous equation and defines $S:=[\nabla u_1\ \nabla u_2]$, then the problem reads $-\nabla\cdot(\sigma S) = \g 0$. A third example is a classical problem corresponding to the particular case where $S$ is the identity matrix everywhere. In this case, the problem reads $-\nabla\mu =  \g f$ which is the inverse gradient problem.

The properties of the gradient operator $\nabla:L^2(\Omega)\to H^{-1}(\Omega,\R^d)$ and its discretization have been extensively studied in particular in the context of fluid dynamics and some tools developed in this framework are useful to treat our more general problem. For the reader convenience, let us recall the most important property which ensures the existence of a bounded left-inverse.

Hence, in the case where $S$ is the identity matrix everywhere, \emph{i.e.} $T:=-\nabla$, the operator $T$ is known to be a closed range operator from $L^2(\Omega)$ to $H^{-1}(\Omega,\R^d)$ if $\Omega$ is a Lipschitz domain (see \cite[p.99]{tartar2006introduction} and references within).  One can write
\begin{align}\nonumber
\norm{q}{L^2(\Omega)} \leq C \norm{\nabla q}{H^{-1}(\Omega)} \quad \forall q\in L^2_0(\Omega),
\end{align}
where $C>0$. The norm of the pseudo-inverse of the gradient in $H^{-1}(\Omega,\R^d)$ is closely related with the \emph{inf-sup} condition of the divergence:

\begin{align}\label{eq:betagrad}
\beta := \inf_{q\in L^2_0(\Omega)} \sup_{\g v\in H^1_0(\Omega,\R^d)} \frac{\int_\Omega (\nabla\cdot \g v)q }{\norm{\g v}{H^{1}_0(\Omega)}\norm{q}{L^2(\Omega)}}>0
\end{align}
Indeed, we have $C=1/\beta$. Since the closed-range property of the gradient is equivalent to the surjectivity of the divergence in $L^2_0(\Omega)$, the study of behavior of $\beta$ is an important step  in establishing  the well-posedness and stability of the Stokes problem \cite[Chap. I, Theorem 4.1]{girault1986p}. The constant $\beta$ is also known as the \emph{LBB} constant (for Ladyzhenskaya-Babuska-Brezzi). It is well known that in general, the constant $\beta$ may not behave well in finite element spaces, and may vanish when the mesh size goes to zero. More precisely, if one considers discrete spaces $M_h\subset L^2(\Omega)$ and $V_h \subset H^{1}_0(\Omega,\R^d)$ with discretization parameter $h>0$, the associated discrete \emph{inf-sup} constant given by
\begin{align}\nonumber
\beta_h :=  \inf_{\substack{q\in M_h\\ q\perp 1}} \sup_{\g v\in V_h}  \frac{\int_\Omega (\nabla\cdot \g v)q }{\norm{\g v}{H^{1}_0(\Omega)}\norm{q}{L^2(\Omega)}}\end{align}
may not satisfy the discrete \emph{inf-sup} condition $\forall h>0, \beta_h\geq \beta^* >0$. Pairs of finite element spaces that satisfy the discrete \emph{inf-sup} condition are known as \emph{inf-sup} stable elements and play an important role in the stability of the Galerkin approximation for the Stokes problem. We refer to \cite{bernardi2016continuity} for more details on the \emph{inf-sup} constant of the gradient and its convergence.

\subsection{Main results}

 Inspired by this approach, we introduce a generalization of the \emph{inf-sup} constant and a corresponding definition of the discrete \emph{inf-sup} constant that are suitable for operators of type \eqref{eq:T1} in particular. A major difference with the classical definition of the \emph{inf-sup} constant of the gradient is that, here, the operator $T$ may contain measurement noise and may have a trivial null space. 
 
 In a general framework, consider $T\in\cL(M,V')$ where $M$ and $V$ are two Hilbert spaces. The problem $T\mu=\g f$ is approached by a finite dimensional problem $T_h\mu_h=\g f_h$ where $T\in\cL(M_h,V_h')$ and $M_h$, $V_h$ approach $M$ and $V$ respectively.
  
 The first main goal of this work is to provide a stability condition with respect to the $M$-norm for the discrete problem based on the associated discrete \emph{inf-sup} constant. We consider the stability with respect to both the noise and the interpolation error on the right-hand side $\g f$ and on the operator $T$ itself. The case $\g f=\g 0$ corresponds to a null space identification problem and the condition $\norm{\mu}{M}=1$ is added. As $T$ may have a null space of dimension one, the stability condition when $\g f\neq \g 0$ is only proved in an hyperplane of $M$ (the orthogonal complement of the approximated null space). The uniqueness of the reconstruction of $\mu$ is then obtained up to a scalar constant. 
Moreover,  we provide quantitative error estimates. They depends on the  discrete \emph{inf-sup} constant and can be explicitly computed in all practical situations dealing with experimental data. These estimates allow for a control of the quality of the reconstruction in the pair of approximation spaces $(M_h,V_h)$ directly from the noisy interpolated data. 
The behavior of the discrete \emph{inf-sup} constant with respect to the discretization parameter $h$ gives a practical criterion for the convergence of $\mu_h$ towards $\mu$.

The present paper is closely linked to the sensitivity analysis and discretization analysis for the Moore-Penrose generalized inverse of $T$ when $T$ is a closed range  operator. There exist a vast litterature on this subject (see \cite{ben2003generalized,ding1996perturbation,yang2010some,huang2012stable} and references herein) as well as on the finite dimensional interpolation of the generalized inverse \cite{du2008finite}.
However, there are fundamental differences between the present work and the existing literature. First, we do not know here whether the operator $T$ has closed range. Second, we perform a sensitivity analysis of the left inverse of $T\in \cL\left(M,V'\right)$ under perturbations that are controlled in a weaker norm. More precisely, perturbations are controlled here in $\cL(E,V')$ where $E\subset M$ is a  Banach space dense in $M$. This might seem a technical issue but it is mandatory if one wants to work with discontinuous parameters $\mu$ and $S$. This choice is motivated by the applications in bio-medical imaging where, in most cases, the biological tissues exhibit discontinuities in their physical properties.  For instance, in the linear elasticity inverse problem (see equation \eqref{eq:elastic}) the matrix $S=2 \cE(\g u)$  has the same surfaces of discontinuities than the shear modulus of the medium and cannot be approached in $L^\infty(\Omega,\R^{d\times d})$ by smooth functions. This leads to perturbations of $T$ in  $\cL\left(L^\infty(\Omega),H^{-1}(\Omega,\R^d)\right)$ instead of  $\cL\left(L^2(\Omega),H^{-1}(\Omega,\R^d)\right)$. More details and examples are given in Subsection \ref{sub:interp}.

\subsection{Outline of the paper}

The article is organized as follows: 
In Section \ref{sec:discretisation}, we describe the Galerkin approximation of the problem \eqref{eq:pbconti} and define all the approximation errors involved. 
In Section \ref{sec_inf_sup}, we generalize the notion of \emph{inf-sup} constant to any operator $T\in \cL(M,V')$ and we prove in Theorem \ref{theo:usc} the upper semi-continuity of the discrete \emph{inf-sup} constant. This is an asymptotic comparison between the discrete and the \emph{continuous} \emph{inf-sup} constants.
In Section \ref{sec3} we give and prove the main stability estimates (Theorems \ref{theo1}, \ref{theo:alpha} and \ref{theo2}) based on the discrete version of the \emph{inf-sup} constant just defined.  In Section \ref{sec5} we present various numerical inversions, including stability tests and numerical computations of the \emph{inf-sup} constant for various pairs of finite element spaces. We also introduce in this section a pair of finite element spaces based on an hexagonal tilling of the domain $\Omega$. It shows excellent numerical stability properties when compared to some more classical pair of discretization spaces.

\section{Discretization using the Galerkin approach} \label{sec:discretisation} 
 We describe the Galerkin approximation of problem \eqref{eq:pbconti} a give the definitions of the various errors of approximation. 

\subsection{General notations}
In all this work, $M$ and $V$ are two Hilbert spaces with respective inner products denoted $\inner{.,.}M$ and $\inner{.,.}V$. We denote $E\subset M$ a Banach space dense in $M$. The space $V':=\cL(V,\R)$ is the space of the bounded linear forms on $V$ endowed with the operator norm. The duality hook between $V'$ and $V$ is denoted $\inner{.,.}{V',V}$. The space $\cL(M,V')$ is the space of the bounded linear operator from $M$ to $V'$ endowed with the operator norm written $\norm{.}{M,V'}$. For any $T\in\cL(M,V')$, we denote its null space by $N(T)$.

\begin{example}\label{ex:1} In the case of the inverse elastography problem using the operator $T$ defined in \eqref{eq:T1}, we take $M:=L^2(\Omega)$, $V:=H^1_0(\Omega,\R^d)$, $E:=L^\infty(\Omega)$ and so $V'=H^{-1}(\Omega,\R^d)$. Here $H^1_0(\Omega,\R^d)$ is the space of all squared integrable vector-valued fonctions $\g v$ on $\Omega$ such that $\nabla \g v$ is also square integrable and such that its trace on $\partial \Omega$ vanishes.  The space $H^{-1}(\Omega,\R^d)$ is the topological dual of $H^1_0(\Omega,\R^d)$.
\end{example}

\subsection{Spaces discretization and projection}

In order to approach the problem \eqref{eq:pbconti} by a finite dimensional problem, we first approach spaces $M$ and $V$ by finite dimensional spaces.
 
 \begin{definition} For any Banach space $X$, we say that a sequence subspaces $(X_h)_{h>0}$ approaches $X$ if this sequence is asymptotically dense in $X$. That means that for any $x\in X$,
there exists a sequence $(x_h)_{h>0}$ such that $x_h\in X_h$ for all $h>0$ and $\norm{x_h-x}{X}$ converges to zero when $h$ goes to zero. We naturally endow $X_h$ with the restriction of the $X$-norm to make it a normed vector space. 
 \end{definition}

 Consider now two sequences of subspaces $(M_h)_{h>0}$ and $(V_h)_{h>0}$ that approach respectively the Hilbert spaces $M$ and $V$. Remark that $E_h:=E\cap M_h$ is dense in $M_h$ so $E_h=M_h$ for any $h>0$ but $E_h$ is endowed with $E$-norm. 
 
 \begin{example}\label{ex:2} In the case of Example \ref{ex:1}, $M=L^2(\Omega)$ and one can chose $M_h$ as the classical finite element space $\PP^0(\Omega_h)$, \emph{i.e.} the class of piecewise constant functions over a subdivision of $\Omega$ by elements of maximum diameter $h>0$ \cite{girault1986p}.
 \end{example}

 \begin{definition} We denote $\pi_h:M\to M_h$ the orthogonal projection form $M$ onto $M_h$. It naturally satisfies $ \lim_{h\to 0}\norm{\pi_h m-m}{M} = 0$ and $\norm{\pi_h m}{M} \leq \norm{m}{M}$, for all $m\in M$.
 We also denote $p_h:M\backslash N(\pi_h)\to M_h$ the normalized projection form $M$ onto $M_h$ defined by $
 p_h(m):=\frac{\pi_h m}{\norm{\pi_h m}{M}},\quad \forall m\in M,\ \pi_h m\neq 0.$
Note that if $\norm{m}{M}=1$, $p_h(m)$ satisfies $\norm{p_h(m)-m}{M}\leq \sqrt{2}\norm{\pi_hm-m}{M}$.
 \end{definition}
 
 In the following, we will assume that $\pi_h$ is also a contraction for the $E$-norm. That means,
 \begin{equation}\label{eq:Econt}
 \forall m\in E\subset M,\quad \norm{\pi_h m}{E}\leq \norm{m}{E}.
 \end{equation}
This hypothesis is true in the case $E:=L^\infty(\Omega)$, $M:=L^2(\Omega)$ and $M_h:=\PP^0(\Omega_h)$ as in Exemple \ref{ex:2}.

\begin{definition} For any non zero $\mu\in M$, we define its relative error of interpolation onto $M_h$ by
\begin{equation}\nonumber
\e^{\text{int}}_h(\mu):=\frac{\norm{\pi_h\mu-\mu}{M}}{\norm{\mu}{M}}.
\end{equation}
\end{definition}

As the sequence of subspaces $V_h\subset V$ approaches $V$, we define $V_h'$ the space of all linear form over $V_h$ endowed with the norm
\begin{equation}\nonumber
\norm{\g f}{V_h'}:=\sup_{\g v\in V_h}\frac{\inner{\g f,\g v}{V',V}}{\norm{\g v}{V}}.
\end{equation}
Note that $\g f\mapsto \g f|_{V_h}$ defines a natural map from $V'$ onto $V_h'$ and then any $\g f\in V$ naturally defines a unique element $\g f|_{V_h}$ of $V_h'$ (and we continue to call it $\g f$).
Then any non zero right-hand side linear form $\g f\in V'$ is approached by a finite dimensional linear form  $\g f_h\in V_h'$ and we define its relative error of interpolation as follows.

\begin{definition} The relative error of interpolation $\e^{\text{rhs}}_h$ between $\g f\neq \g 0$ and $\g f_h$ is defined by $\displaystyle{
\e^{\text{rhs}}_h:=\frac{1}{\norm{\g f}{V'}}\sup_{\g v\in V_h}\frac{\left<\g f_h-\g f,\g v\right>_{V_h',V_h}}{{\norm{\g v}{V_h}}}}.$
\end{definition}


\subsection{Interpolation of the operator} \label{sub:interp}


We approach the operator $T\in \cL(M,V')$ by a finite dimensional operator  $T_h\in \cL(M_h,V_h')$. The error of approximation is defined as the distance between $T$ and $T_h$ for the $\cL(E_h,V_h')$ norm which is weaker than assuming that the between $T-T_h$ is small in $\cL(M_h,V_h')$. We remind the reader that $E_h:=E\cap M_h$ endowed with the $E$-norm.

\begin{definition} The interpolation error $\e^{\text{op}}_h$ between $T$ and $T_h$ is defined by
 \begin{equation}\nonumber
\e^{\text{op}}_h:=\norm{T_h-T}{E_h,V_h'}:=\sup_{\mu\in E_h}\sup_{\g v\in V_h}\frac{\left<(T_h-T)\mu,\g v\right>_{V_h',V_h}}{\norm{\mu}{E}{\norm{\g v}{V}}}.
\end{equation}
\end{definition}
This error contains both the interpolation error over the approximation spaces and the possible noise in measurements used to build $T_h$.

\begin{remark} The reason of the choice of norms comes from the main application where $M:=L^2(\Omega)$, $E:=L^\infty(\Omega)$, $V:=H^1_0(\Omega,\R^d)$ and $T\mu:=-\nabla\cdot(\mu S)$ with $S\in L^\infty(\Omega,\R^{d\times d})$. This operator is approached by $T_h\mu:=-\nabla\cdot(\mu S_h)$ where $S_h$ is a discrete and possibly noisy version of $S$. In this case, the interpolation error $S_h-S$ is expected to be small in $L^2(\Omega,\R^{d\times d})$ but not in $L^\infty(\Omega,\R^{d\times d})$. This conduces to small interpolation error $\e^{\text{op}}_h$ thanks to the control
\begin{equation}\label{eq:Th-T}
\norm{(T_h-T)\mu}{H^{-1}(\Omega)}\leq \norm{S_h-S}{L^2(\Omega)}\norm{\mu}{L^\infty(\Omega)},\quad\forall\mu\in M_h.
\end{equation}
but $T_h-T$ as no reason to be small in $\cL(M_h,V_h')$ (See example \ref{ex:discontinous}). This definition of $\e^{\text{op}}_h$ matches well practical situations like medical imaging for instance where $S$ might be a discontinuous map with \emph{a priori} unknown surfaces of discontinuity. Therefore it makes sense to consider $S_h-S$ small in $L^2(\Omega,\R^{d\times d})$ but not in $L^\infty(\Omega,\R^{d\times d})$. The next example \ref{ex:discontinous} below explains this situation in dimension one.  
\end{remark}

\begin{example}\label{ex:discontinous} In dimension one, take $\Omega:=(-1,1)$, $M=L^2(\Omega)$, $E=$ $L^\infty(\Omega)$ and $V=H^1_0(\Omega)$. Take $S\in L^\infty(\Omega)$ and define $T\mu:=-(\mu S)'$. Fix $h>0$ and consider any uniform subdivision $\Omega_h\subset\Omega$ of size $h$ containing the segment $I_h:=(-h/2,h/2)$ (hence $0$ is not a knot). Define the interpolation spaces $M_h:=\mathbb{P}^0(\Omega_h)$, $V_h:=\mathbb{P}^1_0(\Omega_h)$.
Chose $S= 1+\chi_{(0,1)}$ and $S_h=1+\chi_{(\frac h2,1)}\in M_h$ and $T_h\mu:=-(\mu S_h)'$. An explicit computation gives
\begin{equation}\nonumber
\norm{S_h-S}{L^2(\Omega)}^2=\frac{h}{2}\quad\text{ i.e. }\quad \norm{S_h-S}{L^2(\Omega)}=\cO\left(\sqrt{h}\right).
\end{equation}
Thanks to \eqref{eq:Th-T}, we also get that $\norm{T_h-T}{E_h,V_h'}=\cO\left(\sqrt{h}\right)$. 

Consider now the sequence $\mu_h=h^{-1/2} \chi_{I_h}$ which satisfies $\norm{\mu_h}{L^2(\Omega)}=1$ and a basis test function $v_h\in V_h$ supported in $[-h/2,3h/2]$ and such that $v_h(h/2)=1$. It satisfies $\norm{v_h}{H^1_0(-1,1)}=\sqrt{2/h}$. We can write
\begin{equation}\nonumber
\inner{-(\mu_h(S_h-S))',v_h}{H^{-1},H^1_0}=\int_{I_h}\mu_h(S_h-S)v_h'=h^{-1/2},
\end{equation}
hence
\begin{equation}\nonumber
\sup_{v\in V_h}\frac{\inner{-(\mu_h(S_h-S))',v}{H^{-1},H^1_0}}{\norm{v}{H^1_0(-1,1)}}\geq \frac{\inner{-(\mu_h(S_h-S))',v_h}{H^{-1},H^1_0}}{\norm{v_h}{H^1_0(-1,1)}}=\frac{\sqrt 2}{2},
\end{equation}
and then $\norm{T_h-T}{M_h,V_h'}\geq \frac{\sqrt 2}{2}$. As a consequence $T_h-T$ is not getting small for the $\cL(M_h,V_h')$-norm. 
\end{example}

\section{The generalized \emph{inf-sup} constant}\label{sec_inf_sup}

In this section we generalize the notion of \emph{inf-sup} constant to any operators $T$ in $\cL(M,V')$. Let us first define three useful constants for such operators.

\begin{definition} For any $T \in \cL(M,V')$, we call
\begin{equation}\nonumber
\alpha(T):=\inf_{\mu\in M}\sup_{\g v\in V} \frac{\left<T\mu,\g v\right>_{V',V}}{\norm{\mu}{M}\norm{\g v}{V}}\quad\text{and}\quad \rho(T):=\sup_{\mu\in M}\sup_{\g v\in V} \frac{\left<T\mu,\g v\right>_{V',V}}{\norm{\mu}{M}\norm{\g v}{V}}.
\end{equation}
we also call $\delta(T):=\sqrt{\rho(T)^2-\alpha(T)^2}$.
\end{definition}

We now extend the notion of \emph{inf-sup} constant of the gradient operator to any
operators of $\cL(M,V')$. As the existence of a null space of dimension one is not guaranteed,
we first propose this very general definition of the generalized \emph{inf-sup} constant called $\beta(T)$. 

\subsection{Definition and properties}

\begin{definition}\label{def:infsup} The \emph{inf-sup} constant of direction $e\in M$, $e\neq 0$ of the operator $T\in\cL(M,V')$ is the non-negative number
\begin{equation}\nonumber
\beta_e(T):=\inf_{\substack{ \mu\in M\\ \mu\perp e}}\sup_{\g v\in V}\frac{\left<T\mu,\g v\right>_{V',V}}{\norm{\mu}{M}\norm{\g v}{V}}.
\end{equation}
The generalized \emph{inf-sup} constant of $T$ is now defined by
\begin{equation}\nonumber
\beta(T):=\sup_{\substack{ e\in M\\ \norm{e}{M}=1}}\beta_e(T).
\end{equation}
\end{definition}

It is mandatory here to show that this definition indeed extends the classic definition of the \emph{inf-sup} constant known for $\nabla$-type operators (with a null space of dimension one).

\begin{proposition}\label{prop:betacomp} Let $T\in\cL(M,V')$ and $z\in M$ such that $\norm{z}{M}=1$ and $\norm{T\, z}{V'}^2\leq \alpha(T)^2+\e^2$ for some $\e\geq 0$. We have
\begin{equation}\nonumber
\beta_z(T)^2\leq \beta(T)^2\leq\beta_{z}(T)^2+\e(\delta(T)+\e).
\end{equation}
In case where $\e=0$, it implies that $\beta(T)=\beta_{z}(T)$.
\end{proposition}

The proof of this result uses the self-adjoint operator $S_T\in \cL(M)$ canonically associated with $T$. 

\begin{lemma}\label{lem:comp} For any $T \in \cL(M,V')$, there exists $S_T\in\cL(M)$ self-adjoint positive semi-definite such that for any $\mu\in M$, $\norm{T\mu}{V'}^2=\inner{S_T\mu,\mu}{M}$.
\end{lemma}

\begin{proof} Call $\Phi:V'\to V$ the Riesz isometric identification defined by $\inner{\Phi f,\g v}{V}=\inner{f,\g v}{V',V}$ for any $\g f\in V'$, $v\in V$. Call also $T^*:V\to H$ the adjoint operator of $T$. We have for any $\mu\in M$,\begin{equation}\nonumber
\norm{T\mu}{V'}^2=\norm{\Phi T \mu}{V}^2=\inner{T\mu,\Phi T\mu}{V',V}=\inner{\mu,T^*\Phi T\mu}M=\inner{S_T\mu,\mu}M.
\end{equation}
where $S_T:=T^*\Phi T:M\to M$ is a self-adjoint positive semi-definite operator.
\end{proof}

\begin{proof} (of Proposition \ref{prop:betacomp}) The first inequality comes from the definition of $\beta(T)$. For the second, take $e\in M$ of norm one and consider $m\in E\cap \{z\}^\perp$ of norm one. If $e\perp z$ then $z\in \{e\}^\perp$ and immediately $\beta_e(T)^2\leq \norm{T\, z}{V'}^2 \leq \alpha(T)^2+\e^2\leq \beta_{z}(T)^2+\e(\delta(T)+\e)$.

Suppose now that $\left<e,z\right>_M\neq 0$. Consider $a=-\left<m,e\right>_M/\left<z,e\right>_M$ and $\mu:=a z+m$. It is clear that $\mu\in \{e\}^\perp$ and $\norm{\mu}{M}^2=a^2+1$. Using Lemma \ref{lem:comp}, we write
\begin{equation}\nonumber
\begin{aligned}
\norm{T\mu}{V'}^2 &=\inner{S_T\mu,\mu}M=a^2\inner{S_Tz,z}M+2a\inner{S_Tz,m}M+\inner{S_Tm,m}M\\
&= a^2\norm{T\, z}{V'}^2+2a\inner{S_Tz,m}M+\norm{Tm}{V'}^2\\
&\leq (1+a^2)\norm{Tm}{V'}^2 + a^2\e^2+ 2|a|\left|\inner{S_Tz,m}M\right|.
\end{aligned}
\end{equation}
Using Proposition \ref{prop:app1} we bound $\left|\inner{S_Tz,m}M\right|$ by $\e \delta(T)$ and then
\begin{equation}\nonumber\begin{aligned}
\frac{\norm{T\mu}{V'}^2}{\norm{\mu}{M}^2} &\leq\norm{Tm}{V'}^2+\e^2+\e \delta(T)\\
\inf_{\substack{ \mu\in E\\ \mu\perp e}}\frac{\norm{T\mu}{V'}^2}{\norm{\mu}{M}^2} &\leq\norm{Tm}{V'}^2+\e(\delta(T)+\e)\\
\beta_e(T)^2&\leq \norm{Tm}{V'}^2+\e(\delta(T)+\e).\\
\end{aligned}\end{equation}
This last statement is true for any  $m\in M\cap \{z\}^\perp$ of norm one so we can take the infimum over $m$ to get $\beta_e(T)^2\leq \beta_{z}(T)^2+\e(\delta(T)+\e)$. We conclude now by taking the supremum over $e$.
\end{proof}

As a consequence of Proposition \ref{prop:betacomp}, the generalized \emph{inf-sup} constant has a simpler formula in the case of an operator with trivial null space. 

\begin{corollary}  \label{cor:def1} If $N(T)\neq \{0\}$, consider any $z\in  N(T)$ such that $\norm{z}M=1$. Then we have $\beta(T)=\beta_z(T)$.
\end{corollary}

If $T=\nabla$, the classic definition of $\beta(\nabla)$ given in \eqref{eq:betagrad} matches the definition \ref{def:infsup}.

\begin{remark} This corollary leads to an alternative definition of $\beta(T)$ which does not depend on the choice of $z$ in $ N(T)$ (even for a dimension greater than one). Moreover, we see that $\beta(T)>0$ implies $\dim N(T)=1$. 
\end{remark}

It is possible to extend a little this corollary to a class of operators with trivial null space if the infimum value of the operator on the unit sphere is reached.

\begin{corollary}\label{cor:def2} If there exists $z\in M$ such that $\norm{z}{M}=1$ and $\norm{T z}{V'}=\alpha(T)$, Then we have $\beta(T)=\beta_z(T)$.
\end{corollary}

\begin{remark} This corollary leads to an alternative definition of $\beta(T)$ which does not depend on the choice of $z$ and extends the definition \ref{cor:def1}. Moreover the condition is fulfilled in particular if $T$ is a finite rank or finite dimensional operator.
\end{remark}

If the infimum value $\alpha(T)$ is not reached on the unit sphere, we keep the general definition \ref{def:infsup}.

\subsection{Discrete \emph{inf-sup} constant} 

The different constants related to the approximated operator $T_h\in \cL(M_h,V_h')$ comes from the same definition than for the operator $T\in\cL(M,V')$. Simply remark that as $T_h$ is a finite dimensional operator, the infimum in 
\begin{equation}\label{eq:ah}
\alpha(T_h):=\inf_{\mu\in M_h}\sup_{\g v\in V_h} \frac{\left<T_h\mu,\g v\right>_{V_h',V_h}}{\norm{\mu}{M}\norm{\g v}{V}}
\end{equation}
is reached by a direction $z_h\in M_h$ such that $\norm{z_h}{M}=1$. This means that $\norm{T_hz_h}{V_h'}=\alpha(T_h)$. As a consequence, following Corollary \ref{cor:def2}, the \emph{inf-sup} constant of $T_h$ is given by
\begin{equation}\label{eq:bh}
\beta(T_h):=\inf_{\substack{ \mu\in M_h\\ \mu\perp z_h}}\sup_{\g v\in V_h} \frac{\left<T_h\mu,\g v\right>_{V_h',V_h}}{\norm{\mu}{M}\norm{\g v}{V}}.
\end{equation}

This discrete \emph{inf-sup} constant is the key element to establish the stability of the discrete inverse problem and as we will see, its behaviors when $h\to 0$ will determine the convergence of the solution of the discrete problem to the exact solution. In a similar way than for the classic \emph{inf-sup} constant, the behavior of the discrete $\inf-\sup$ constant $\beta(T_h)$ can be catastrophic in the sense that it can vanish to zero if $h\to 0$. This strongly depends on the choice of interpolation pair of spaces $(M_h,V_h)$. For instance, if the discrete operator $T_h:M_h\to V_h'$ is under determinate, one may have $\beta(T_h)=0$. In a same manner than in \cite{costabel2015inf}, we give a definition of the discrete \emph{inf-sup} condition.

\begin{definition}\label{def:difc} We say that the sequence of operators $(T_h)_{h>0}$  satisfies the discrete \emph{inf-sup}  condition if there exists $\beta^*>0$ such that 
\begin{equation}\label{eq:discreteinfsup}
\beta^*\leq \beta(T_h),\quad \forall h>0.
\end{equation} 
\end{definition}

\begin{remark} In this work, we do not prove that the discrete \emph{inf-sup}  condition is satisfied by some specific choices of discretized operators $T_h:M_h\to V_h'$. We mention it here as a condition for uniform stability with respect to $h$, (see Theorems \ref{theo1} \ref{theo2}). We only aim at giving discrete stability estimates that involves $\beta(T_h)$ for a fixed $h>0$. 
 
\end{remark}

\subsection{Upper semi-continuity of the \emph{inf-sup} constant}\label{sec4}

A legitimate question about the discrete \emph{inf-sup} constant is to know if it can be greater that the continuous \emph{inf-sup} constant if the discretization spaces are well chosen. Inspired by a classic result on the discrete \emph{inf-sup} of the divergence that can be found in \cite{costabel2015inf} for instance, we state and prove in this subsection that the discrete \emph{inf-sup} constant is upper semi-continuous when $h\to 0$. This concludes that the discrete \emph{inf-sup} constant $\beta(T_h)$ is always asymptotically worse than the continuous \emph{inf-sup} constant $\beta(T)$.

\begin{theorem}[Upper semi-continuity]\label{theo:usc} If $\e^\text{op}_h\to 0$ when $h\to 0$, then 
\begin{equation}\nonumber
\limsup_{h\to 0}\alpha(T_h)\leq \alpha(T).
\end{equation}
Moreover, if the problem $T\, z=\g 0$ admits a solution $z\in E$ with $\norm{z}{M}=1$ and if the sequence $(T_h)_{h>0}$ satisfies the discrete \emph{inf-sup} condition (see Definition \ref{def:difc}), then
\begin{equation}\nonumber
\limsup_{h\to 0}\beta(T_h)\leq \beta(T).
\end{equation}
\end{theorem}

\begin{remark} This result is useful to understand that no discretization can get a better stability constant than $\beta(T)$. The question of the convergence of $\alpha(T_h)$ and $\beta(T_h)$ toward respectively $\alpha(T)$ and $\beta(T)$ is not treaded here; it is clearly not a simple question. It is already known as a difficult issue concerning \emph{inf-sup} constant of the gradient operator. See \cite{bernardi2016continuity} for more details about this question.  
\end{remark}

\begin{remark} An interesting consequence of this result is that, in case of an operator $T$ with non-trivial null space, the fact that $(T_h)_{h>0}$ satisfies the discrete \emph{inf-sup} condition implies that $\beta(T)>0$ which means that $T$ has closed range. It could be used to prove the closed range property for some operators. For instance, to our knowledge, the minimal conditions on $S\in L^\infty(\Omega,\R^{d\times d})$ that make $T:\mu\mapsto -\nabla\cdot(\mu S)$ a closed range operator are not known. 
\end{remark}

\begin{proof}(of Theorem \ref{theo:usc})  First define the sequence of set
\begin{equation}\nonumber
C_h:=\left\{\mu\in M_h\ |\ (\e_h^{\text{op}})^{1/2}\norm{\mu}{E}\leq \norm{\mu}{M}\right\}.
\end{equation}
For any $h>0$ and $\mu\in C_h$ we get
\begin{equation}\label{eq:tech1}\begin{aligned}
\norm{T_h\mu}{V_h'} &\leq \norm{T\mu}{V_h'} + \norm{(T_h-T)\mu}{V_h'}\leq \norm{T\mu}{V'} +\e_h^{\text{op}}\norm{\mu}{E}\\
&\leq \norm{T\mu}{V'} +(\e_h^{\text{op}})^{1/2}\norm{\mu}{M}.
\end{aligned}\end{equation}
Hence 
\begin{equation}\nonumber\begin{aligned}
\alpha(T_h) &\leq \frac{\norm{T\mu}{V'}}{\norm{\mu}{M}} + (\e_h^{\text{op}})^{1/2},\quad \forall\mu\in C_h\\
\alpha(T_h) &\leq \inf_{\mu\in C_h}\frac{\norm{T_h\mu}{V_h'}}{\norm{\mu}{M}} +  (\e_h^{\text{op}})^{1/2}.\\
\end{aligned}\end{equation}
This is true for any $h>0$ so $\displaystyle\limsup_{h\to 0}\alpha(T_h) \leq \limsup_{h\to 0}\inf_{\mu\in C_h}  \frac{\norm{T\mu}{V'}}{\norm{\mu}{M}}.$

As proposition \ref{prop:lim1} shows that $\lim_{h\to 0}C_h=M$ in the sense of Definition \ref{def:setlimit}, using that $T$ is continuous over the sphere $\{\mu\in M\ |\ \norm{\mu}{M}=1\}$ we can use Proposition \ref{prop:limsup} that says 
\begin{equation}\nonumber
\limsup_{h\to 0}\inf_{\mu\in C_h}  \frac{\norm{T\mu}{V'}}{\norm{\mu}{M}}\leq \inf_{\mu\in M}  \frac{\norm{T\mu}{V'}}{\norm{\mu}{M}}=\alpha(T)
\end{equation}
which gives the first result. 

For the second result, consider the sequence $(z_h)_{h>0}$ that satisfies $\norm{z_h}{M}=1$ and $T_hz_h=\alpha(T_h)$. Then $\beta(T_h)=\beta_{z_h}(T_h)$. For any $h>0$ and $\mu\in C_h\cap \{z_h\}^\perp$, similarly to \eqref{eq:tech1}, we get
\begin{equation}\nonumber
\norm{T_h\mu}{V_h'}\leq \norm{T\mu}{V'} +(\e_h^{\text{op}})^{1/2}\norm{\mu}{M},
\end{equation}
and then by definition of $\beta(T_h)$,
\begin{equation}\nonumber\begin{aligned}
\beta(T_h) &\leq \frac{\norm{T\mu}{V'}}{\norm{\mu}{M}} + (\e_h^{\text{op}})^{1/2},\quad \forall\mu\in C_h\cap \{z_h\}^\perp\\
\beta(T_h) &\leq \inf_{\mu\in C_h\cap \{z_h\}^\perp}\frac{\norm{T_h\mu}{V_h'}}{\norm{\mu}{M}}.\\
\end{aligned}\end{equation}
This is true for any $h>0$ so we deduce
\begin{equation}\nonumber
\limsup_{h\to 0}\beta(T_h) \leq \limsup_{h\to 0}\inf_{\mu\in C_h\cap \{z_h\}^\perp}  \frac{\norm{T\mu}{V'}}{\norm{\mu}{M}}.
\end{equation}
Now as Theorem \ref{theo1} says that the sequence $z_h$ converges to $z$ in $M$ and Proposition \ref{prop:wc} gives that $\lim_{h\to 0}C_h\cap \{z_h\}^\perp=M\cap\{z\}^\perp$, we can use Proposition \ref{prop:limsup} that says
\begin{equation}\nonumber
\limsup_{h\to 0}\inf_{\mu\in C_h\cap \{z_h\}^\perp}  \frac{\norm{T\mu}{V'}}{\norm{\mu}{M}}\leq \inf_{\mu\in M\cap \{z\}^\perp}  \frac{\norm{T\mu}{V'}}{\norm{\mu}{M}} = \beta_z(T)=\beta(T)
\end{equation}
which gives the second result. 
\end{proof}

\section{Error estimates}\label{sec3}

In this section, we state and prove the error estimates that are stability estimates for the approximated problem $T_h\mu_h=\g f_h$.

\subsection{Error estimate in the case $\g f=\g 0$}

\begin{theorem}[Error estimate in the case $\g f=\g 0$]\label{theo1} Consider $T\in\cL(M,V')$ and let $z\in E$ be a solution of $T\, z=\g 0$  with $\norm{z}{M}=1$ that satisfies $\e^{\text{int}}_h(z)\leq 1/2$. Fix $r$ such that $ \norm{z}{E}\leq r$ and consider $z_h\in M_h$ a solution of 
\begin{equation}\label{eq:zh}
\norm{T_hz_h}{V_h'} = \alpha(T_h)\quad\text{with}\quad\norm{z_h}{M}=1\quad\text{and}\quad \left<z_h,z\right>_M\geq 0.
\end{equation}
If $\beta(T_h)>0$ we have 
\begin{equation}\nonumber
\begin{aligned}
\norm{z_h-p_h(z)}{M}\leq \frac{4}{\beta(T_h)} \big(\sqrt{2}\, r\, \e^{\text{op}}_h+2\rho(T)\e^{\text{int}}_h(z)\big).\\
\end{aligned}
\end{equation}
Moreover, if $\e^{\text{op}}_h\to 0$ and $(T_h)$ satisfies the discrete \emph{inf-sup} condition \eqref{eq:discreteinfsup}, then $\Vert z_h - z\Vert_M \to 0$.
\end{theorem}
\begin{remark} 
\
\begin{enumerate}
\item Note that if $\e^{\text{op}}_h\to 0$, since $\alpha(T)  = 0$, we have, from Theorem \ref{theo:usc}, that $\alpha(T_h) \to 0$. Moreover, if the discrete inf sup condition (equation \eqref{eq:discreteinfsup}) is satisfied, then $z_h$  is defined uniquely.
\item It is necessary to have the priori bound $\Vert z\Vert_E\leq r$ to overcome the fact that $T_h-T$ is controlled in $\cL(E_h,V_h')$ but not in $\cL(M_h,V_h')$. See section \ref{sub:interp} for more details.
\item In the framework of the inverse elastography problem, the hypothesis $z\in E:=L^\infty(\Omega)$ is not restrictive as physical parameters of biological tissues have bounded values with some known \emph{a piori} bounds. 
\item The normalized projection $p_h(z)$ of $z$ is the best possible approximation of $z$ in $M_h$ with the constraint of norm one. 
\item Problem \eqref{eq:zh} admits a solution $z_h$ as $T_h$ is a finite dimensional operator. The condition $\left<z_h,z\right>_M\geq 0$ is only here to chose between the two solutions $z_h$ and $-z_h$ and is not of crucial importance. 
\item This result provides a quantitative error estimate as $\beta(T_h)$ can be computed from $T_h$ as the second smallest singular value (see Subsection\ref{sub:matrix}) and all the error terms on the right-hand side can be estimated (at least an upper bound can be given).

\end{enumerate}
\end{remark}

Before giving the proof of Theorem \ref{theo1}, we first establish and prove a more general result. 

\begin{proposition}\label{prop:theo1} Consider $T_1\in\cL(M,V')$ let $z_1\in E$ be a solution of 
\begin{equation}\nonumber
\norm{T_1\, z_1}{V'}\leq \alpha(T_1)  + \e_1\quad \text{with}\quad \norm{z_1}{M}=1
\end{equation}
where $\e_1\geq 0$. Fix $r\geq \norm{z_1}{E}$. For any $T_2\in \cL(M,V')$, consider a solution $z_2\in E$ of 
\begin{equation}\nonumber
\norm{T_2\, z_2}{V'}\leq \alpha(T_2)  + \e_2\quad \text{with}\quad\norm{z_2}{M}=1\quad\text{ and }\quad \left<z_1,z_2\right>_M\geq 0.
\end{equation}
If $\beta_{z_2}(T_2)>0$ we have $
\norm{z_2-z_1}{M}\leq \frac{\sqrt{2}}{\beta_{z_2}(T_2)} \left(2r\norm{T_2-T_1}{E,V'}+2 \alpha(T_1)+2\e_1+\e_2\right)$
and if $\e_2=0$ this reads $
\norm{z_2-z_1}{M}\leq \frac{\sqrt{2}}{\beta(T_2)} \left(2r\norm{T_2-T_1}{E,V'}+2 \alpha(T_1)+2\e_1\right).$
\end{proposition}

\begin{proof}   Write $z_1= t{z_2} + m$ where $t\in[0,1]$ and $m\perp z_2$. We have that $1=t^2+\norm{m}{M}^2$. Then $z_1-{z_2} = (t-1) z_2 + m$ and so $\norm{z_2-z_1}{M}^2 = 2(1-t)\leq 2(1-t^2)\leq 2\norm{m}{M}^2$. Then $\norm{z_2-z_1}{M}\leq \sqrt{2}\norm{m}{M}$. Now use the definition of $\beta_{z_2}(T_2)$ to write
\begin{equation}\nonumber\begin{aligned}
\beta_{z_2}(T_2)\norm{m}{M} &\leq \norm{T_2 m}{V'} \leq  \norm{T_2 z_1}{V'}+\norm{T_2 z_2}{V'} \leq   \norm{T_2 z_1}{V'}+\alpha(T_2)+{\e_2}\\ &\leq   2\norm{T_2 z_1}{V'}+{\e_2}\\
\end{aligned}\end{equation}
and remark that $\norm{T_2z_1}{V'}\leq \norm{(T_2-T_1)z_1}{V'}+\norm{T_1z_1}{V'}\leq r\norm{T_2-T_1}{E,V'} +\norm{T_1z_1}{V'}$ which implies that
\begin{equation}\nonumber
\norm{T_2 z_1}{V'}\leq  r\norm{T_2-T_1}{E,V'} + \alpha(T_1)+\e_1.
\end{equation}
We deduce that $
\beta_{z_2}(T_2)\norm{m}{M}\leq 2r\norm{T_2-T_1}{E,V'}+2 \alpha(T_1)+2\e_1+\e_2 $ 
and then $
\norm{z_2-{z_1}}{M}\leq \frac{\sqrt{2}}{\beta_{z_2}(T_2)} \left(2r\norm{T_2-T_1}{E,V'}+2 \alpha(T)+2\e_1+\e_2\right).$
\end{proof}
We now give the proof of Theorem \ref{theo1}:
\begin{proof} First remark that the infimum  in \eqref{eq:zh} is reached here because $T_h$ is a finite dimensional operator. Consider $T|_{M_h}:M_h\to V_h'$ and call $ g_h:=Tp_h(z)$. This quantity is small in $V_h'$ as
 \begin{equation}\nonumber
\begin{aligned}
\norm{\g g_h}{V_h'} &=\norm{Tp_h(z)}{V_h'}=\norm{T(p_h(z)-z)}{V_h'}\leq \norm{T}{M,V'}\norm{p_h(z)-z}{M}\\
 &\leq \sqrt{2}\rho(T)\e^{\text{int}}_h(z).
\end{aligned}
\end{equation}
From this, we deduce that $\alpha(T|_{M_h})\leq \sqrt{2}\rho(T)\e^{\text{int}}_h(z)$ and that $p_h(z)$ is solution of
\begin{equation}\nonumber
\norm{T|_{M_h}p_h(z)}{V_h'}\leq \alpha(T|_{M_h})  + \e\quad \text{with}\quad \norm{p_h(z)}{M}=1,
\end{equation}
with $\e=\sqrt{2}\rho(T)\e^{\text{int}}_h(z)$. Due to Hypothesis \eqref{eq:Econt} and $\e^{\text{int}}_h(z)\leq 1/2$ we have 
\begin{equation}\nonumber
\norm{p_h(z)}{E}=\frac{\norm{\pi_hz}{E}}{\norm{\pi_hz}{M}}\leq 2\frac{\norm{z}{E}}{\norm{z}{M}}\leq 2r.
\end{equation}
 Applying now Proposition \eqref{prop:theo1} on operators $T_1=T|_{M_h}$ and $T_2=T_h$ both in $\cL(M_h,V_h')$ with $z_1 =p_h(z)$, $z_2=z_h$, $\e_1=\e$ and $\e_2=0$. We get
\begin{equation}\nonumber
\begin{aligned}
\norm{z_h-p_h(z)}{M}&\leq \frac{\sqrt{2}}{\beta(T_h)} \left(4r\, \e^{\text{op}}_h+2 \alpha(T|_{M_h})+2\e\right)\\
&\leq \frac{\sqrt{2}}{\beta(T_h)} \left(4r\, \e^{\text{op}}_h+4 \sqrt{2}\rho(T)\e^{\text{int}}_h(z)\right)\\
&\leq \frac{4}{\beta(T_h)} \left(\sqrt{2}\, r\, \e^{\text{op}}_h+2\rho(T)\e^{\text{int}}_h(z)\right).\\
\end{aligned}
\end{equation}
For the convergence, the additional hypothesis give the convergence of the right-hand side. We use that $p_h(z)\to z$ to conclude. 
\end{proof}


\subsection{Error estimates in the case $\g f\neq \g 0$}

We give and prove a first stability result based on the constant $\alpha(T_h)$. 
\begin{theorem}[Error estimate using $\alpha(T_h)$]\label{theo:alpha} Consider $\mu\in E$ a solution of  $T\mu=\g f$  with $\g f\neq 0$ and which satisfies $\e^{\text{int}}_h(\mu)\leq 1/2$. Fix $r>0$ such that $\norm{\mu}{E}\leq r\norm{\mu}{M}$. Consider now $\mu_h\in M_h$ a solution of $ \mu_h=\argmin_{m\in M_h} \norm{T_hm-\g f_h}{V_h'}.$
If $\alpha(T_h)>0$, we have  
\begin{equation}\nonumber\begin{aligned}
\frac{\norm{\mu_{h}-\pi_h\mu}{M}}{\norm{\pi_h\mu}{M}} &\leq \frac{4}{\alpha\left(T_h\right)}\left[ r\, \e^{\text{op}}_h + \rho(T)\left(\e^{\text{rhs}}_h+\e^{\text{int}}_h(\mu)\right)\right].\\
\end{aligned}\end{equation}
Moreover, if there exists $\alpha^*>0$ such that $\alpha(T_h)\geq \alpha^*$ for all $h>0$ and if $\e^{\text{op}}_h\to 0$ and $\e^{\text{rhs}}_h\to 0$ when $h\to 0$, we get $\norm{\mu_h- \mu}{M}\to 0$ when $h\to 0$.
 \end{theorem}
 
 \begin{remark}  Note that if  $\alpha(T_h)>0$ for all $h>0$, then $\mu_h$ is uniquely defined and moreover $\e^{\text{op}}_h\to 0$ and if $\alpha(T_h)\geq \alpha_*>0$, Theorem \ref{theo:usc} assures that $\alpha(T)\geq \alpha^*>0$ which guarantee the uniqueness of $\mu$. 
\end{remark}

 \begin{remark} This result makes sense in practice even if $\alpha(T_h)$ goes to zero. Indeed, at a fixed $h>0$, $\alpha(T_h)$ can be computed from $T_h$ as the first singular value and all the error terms on the right-hand side can be estimated (at least an upper bound can be given). It then gives a quantitative error bound on the reconstruction that can be useful no matter with the asymptotic behavior of $\alpha(T_h)$.
 \end{remark}
 
 \begin{proof} First note that from the hypothesis $\e^{\text{int}}_h(\mu)\leq 1/2$ we have that $\norm{\mu}{M}\leq 2\norm{\pi_h\mu}{M}$ and $\norm{\pi_h\mu}{E}\leq \norm{\mu}{E}\leq r\norm{\mu}{M}\leq 2r\norm{\pi_h\mu}{M}$and  $\norm{\g f}{V'}\leq \rho(T)\norm{\mu}{M}$. From the definition of $\alpha(T_h)$ we write 
 \begin{equation}\nonumber\begin{aligned}
 \alpha(T_h)\norm{\mu_h-\pi_h\mu}{M} &\leq \norm{T_h\mu_h-T_h\pi_h\mu}{V_h'}\leq \norm{T_h\mu_h-\g f_h}{V_h'} + \norm{T_h\pi_h\mu-\g f_h}{V_h'}\\
 &\leq 2\norm{T_h\pi_h\mu-\g f_h}{V_h'}\\
  &\leq 2\norm{T\mu-\g f_h}{V_h'} + 2\norm{T\, \pi_h\mu - T\mu}{V_h'}+ 2\norm{(T_h-T)\, \pi_h\mu}{V_h'}\\
  &\leq 2\norm{\g f-\g f_h}{V_h'} + 2\rho(T)\norm{\pi_h\mu-\mu}{M}+ 2 \, \e^{\text{op}}_h\norm{\pi_h\mu }{E}\\
  &\leq 2\e^{\text{rhs}}_h\norm{\g f}{V'} + 2\rho(T)\e^{\text{int}}_h(\mu)\norm{\mu}{M}+ 4r \, \e^{\text{op}}_h\norm{\pi_h\mu }{M}\\
    &\leq 2\rho(T)\left(\e^{\text{rhs}}_h+\e^{\text{int}}_h(\mu)\right)\norm{\mu}{M} + 4r\, \e^{\text{op}}_h\norm{\pi_h\mu }{M}\\
    &\leq 4\left[\rho(T)\left(\e^{\text{rhs}}_h+\e^{\text{int}}_h(\mu)\right)+ r\e^{\text{op}}_h\right] \norm{\pi_h\mu }{M}.
  \end{aligned}\end{equation}
  \
 \end{proof}

We now state and prove the main stability estimate concerning the general problem $T\mu=\g f$ with a non zero right-hand side. This result uses $\beta(T_h)$ which is always better than $\alpha(T_h)$. The price of this change is that the stability estimates only holds in the hyperplane $\{z_h\}^\perp$, where $z_h$ is the vector that minimizes $\norm{T_hz_h}{V_h'}$ on the unit sphere.

\begin{theorem}[Error estimate using $\beta(T_h)$]\label{theo2} Consider $\mu\in E$ a solution of  $T\mu=\g f$  with $\g f\neq 0$ and which satisfies $\e^{\text{int}}_h(\mu)\leq 1/2$. Fix $r>0$ such 
that $\norm{\mu}{E}\leq r\norm{\mu}{M}$. Consider $z_h\in M_h$ a solution of 
\begin{equation}\nonumber
\norm{T_hz_h}{V_h'} = \alpha(T_h)\quad\text{with}\quad\norm{z_h}{M}=1.\end{equation}
Consider now $\mu_h\in M_h$ a solution of
\begin{equation}\label{eq:muh}
\mu_h=\argmin_{\substack{ m\in M_h\\ m\perp z_h}} \norm{T_hm-\g f_h}{V_h'}, \quad\text{with } \mu_h\perp z_h.
\end{equation}
If $\beta(T_h)>0$, there exits $t\in\R$ such that $\mu_{h,t}:=tz_h+\mu_h$ satisfies 
\begin{equation}\nonumber\begin{aligned}
\frac{\norm{\mu_{h,t}-\pi_h\mu}{M}}{\norm{\pi_h\mu}{M}} &\leq \frac{4}{\beta\left(T_h\right)}\left[ r\, \e^{\text{op}}_h + \rho(T)\left(\e^{\text{rhs}}_h+\e^{\text{int}}_h(\mu)\right) +\frac{\alpha\left(T_h\right)}{2}\right].\\
\end{aligned}\end{equation}
 \end{theorem}

\begin{remark} This result has to be used as soon as Theorem \ref{theo:alpha} is irrelevant because $\alpha(T_h)$ is too small. 
It somehow kills the degenerated direction $z_h$ and gives a possibly better estimate for the computed solution up to an unknown 
component in the direction $z_h$.
\end{remark} 

\begin{remark} This result gives also the algorithmic procedure to approach the exact solution $\mu$:
\begin{enumerate}
\item Identify $z_h$ with stability thanks to Theorem \ref{theo1}.
\item Solve the problem \eqref{eq:muh} to identify $\mu_h$.
\item Find the best approximation $tz_h+\mu_h$ by choosing a correct coefficient $t\in\R$ using any additional scalar information on the exact solution such as its mean, its background value, a punctual value, etc\dots
\end{enumerate}
\end{remark} 

\begin{remark} This result provides a quantitative error estimate as $\alpha(T_h)$ and $\beta(T_h)$ can be computed from $T_h$ as the
two first singular values and all the error terms on the right-hand side can be estimated (at least an upper bound can be given).
\end{remark}

 Before giving the proof of this Theorem, let us state and prove an intermediate result.
 
\begin{proposition}\label{prop:5}Consider $T_1\in\cL(M,V')$ $\g f_1\in V'$, $\g f_1\neq 0$ and let $z_1\in E$ be a solution of $T_1\, \mu_1 = \g f_1$. Fix $r>0$ such that $\norm{\mu_1}{E}\leq r\norm{\mu_1}{M}$ and for any $T_2\in \cL(M,V')$, consider a solution $z_2\in E$ of 
\begin{equation}\nonumber
\norm{T_2\, z_2}{V'}\leq \alpha(T_2)  + \e_2\quad\text{and}\quad\norm{z_2}{M}=1
\end{equation}
and consider a solution $\mu_2\in E$ of 
\begin{equation}\nonumber
T_2\, \mu_2 =\g f_2\quad\text{ and }\quad \mu_2\perp z_2.
\end{equation}
If $\beta_{z_2}(T_2)>0$, there exits $t\in\R$ such that $\mu_{2,t}:=tz_2+\mu_2$ satisfies 
\begin{equation}\nonumber\begin{aligned}
\frac{\norm{\mu_{2,t}-\mu_1}{M}}{\norm{\mu_1}{M}} &\leq \frac{1}{\beta_{z_2}(T_2)}\left( \frac{\norm{\g f_2-\g f_1}{V'}}{\norm{\mu_1}{M}}  +   r\norm{T_2-T_1}{E,V'}+\alpha(T_2)+\e_2\right).
\end{aligned}\end{equation}
Moreover, if $\e_2=0$ it reads 
\begin{equation}\nonumber
\frac{\norm{\mu_{2,t}-\mu_1}{M}}{\norm{\mu_1}{M}} \leq \frac{1}{\beta(T_2)}\left(  \frac{\norm{\g f_2-\g f_1}{V'}}{\norm{\mu_1}{M}}  +   r\norm{T_2-T_1}{E,V'}+\alpha(T_2)\right).
\end{equation}
\end{proposition}

\begin{proof} Denote $\mu_{2,t}:=tz_2+\mu_2$ with $t:=\left<\mu,z_2\right>_M$. With this choice, we have that $(\mu_{2,t}-\mu_1)\perp z_2$. From the definition of $\beta_{z_2}(T_2)$, we write 
\begin{equation}\nonumber\begin{aligned}
\beta_{z_2}(T_2)\norm{\mu_{2,t}-\mu_1}{M}  &\leq \norm{T_2\, \mu_{2,t}-T_2\, \mu_1}{V'}\\ &\leq  \norm{T_2\mu_2-T_1\mu_1}{V'}+  |t|\norm{T_2\, z_2}{V'}+ \norm{(T_2-T_1)\mu_1}{V'}\\
&\leq \norm{\g f_2-\g f_1}{V'} +\norm{\mu_1}{M}\left(\alpha(T_2)+\e_2\right) + \norm{T_2-T_1}{E,V'}\norm{\mu_1}{E}.\\
&\leq \norm{\g f_2-\g f_1}{V'} +\norm{\mu_1}{M}\left(\alpha(T_2)+\e_2 + r\norm{T_2-T_1}{E,V'}\right).\\
\end{aligned}\end{equation}
\
\end{proof}

We can now give the proof of Theorem \ref{theo2}.

\begin{proof} (of Theorem \ref{theo2}) Consider $T|_{M_h}:E_h\to V_h'$ and call $\g g_h:=T\pi_h\mu$. Remark that $\norm{\pi_h\mu}{E}\leq \norm{\mu}{E}\leq r\norm{\mu}{M}\leq 2r\norm{\pi_h\mu}{M}$. Applying Proposition \ref{prop:5} to the operators $T_1:=T|_{E_h}$, $T_2:=T_h$ both in $\cL(M_h,V_h')$, with $\g f_1:=\g g_h$, $\g f_2:=T_h\mu_h$ both in $V_h'$ and with $\mu_1:=\pi_h\mu$, $\mu_2:=\mu_h$. We get the existence of $t\in\R$ such that 
 \begin{equation}\nonumber
\frac{\norm{\mu_{h,t}-\pi_h\mu}{M}}{\norm{\pi_h\mu }{M}} \leq \frac{1}{\beta(T_h)}\left( \frac{\norm{T_h\mu_h-\g g_h}{V_h'}}{\norm{\pi_h\mu}{M}}  +   2r\, \e^{\text{op}}_h+\alpha(T_h)\right).
\end{equation}
Now we bound $\norm{T_h\mu_h-\g g_h}{V_h'}$ as follows: 
$$\norm{T_h\mu_h-\g g_h}{V_h'}\leq \norm{T_h\mu_h-\g f_h}{V_h'} + \norm{\g g_h-\g f_h}{V_h'}.$$ To deal with the first term, we define $p:=\pi_h\mu-\left<\pi_h\mu,z_h\right>_Mz_h$ orthogonal to $z_h$. We have
\begin{equation}\nonumber
\begin{aligned}
\norm{T_h\mu_h-\g f_h}{V_h'} &\leq \norm{T_hp-\g f_h}{V_h'}\leq \norm{T_h\pi_h\mu-\g f_h}{V_h'} + \norm{T_h z_h}{V_h'}\norm{\pi_h\mu}{M}\\
&\leq \norm{T\pi_h\mu-\g f_h}{V_h'} + \norm{(T_h-T)\pi_h\mu}{V_h'}+\alpha(T_h)\norm{\pi_h\mu}{M}\\
&\leq \norm{\g g_h-\g f_h}{V_h'}+ \e^{\text{op}}_h\norm{\pi_h\mu}{E}+\alpha(T_h)\norm{\pi_h\mu}{M}\\
&\leq \norm{\g g_h-\g f_h}{V_h'}+ \left(2r\, \e^{\text{op}}_h+\alpha(T_h)\right)\norm{\pi_h\mu}{M}.
\end{aligned}
\end{equation}
Now the second term is bounded as follows:
\begin{equation}\nonumber\begin{aligned}
\norm{\g g_h-\g f_h}{V_h'} &\leq  \norm{\g g_h - \g f}{V_h'} +  \norm{\g f - \g f_h}{V_h'}\leq \norm{T\pi_h\mu- T\mu}{V_h'} +  \e_h^{\text{rhs}}\norm{\g f}{V'}\\
&\leq \rho(T)\e_h^{\text{int}}(\mu)\norm{\mu}{M} +  \rho(T)\e_h^{\text{rhs}}\norm{\mu}{M}\leq \rho(T)\norm{\mu}{M}\left(\e_h^{\text{int}}(\mu) +  \e_h^{\text{rhs}}\right)\\
&\leq 2\rho(T)\norm{\pi_h\mu}{M}\left(\e_h^{\text{int}}(\mu) +  \e_h^{\text{rhs}}\right).\\
\end{aligned}\end{equation}
This last line is true because the hypothesis $\e^{\text{int}}_h(\mu)\leq 1/2$ implies that $\norm{\mu}{M}\leq 2\norm{\pi_h\mu}{M}$. Putting things together, it come that 
\begin{equation}\nonumber
 \frac{\norm{T_h\mu_h-\g g_h}{V_h'}}{\norm{\pi_h\mu}{M}} \leq 4\rho(T)\left(\e_h^{\text{int}}(\mu) +  \e_h^{\text{rhs}}\right) + 2r\, \e^{\text{op}}_h+\alpha(T_h)
\end{equation}
and then
\begin{equation}\nonumber
\frac{\norm{\mu_{h,t}-\pi_h\mu}{M}}{\norm{\pi_h\mu }{M}} \leq \frac{2}{\beta(T_h)}\left[  2\rho(T)\left(\e_h^{\text{int}}(\mu) +  \e_h^{\text{rhs}}\right) + 2r\, \e^{\text{op}}_h+\alpha(T_h)\right].
\end{equation}
\
\end{proof}

\section{Numerical results}\label{sec5}

In this section we provide numerical applications of Theorems \ref{theo1} and \ref{theo2} and we present the general methodology to
numerically approach the solution of the equation \eqref{eq:1} in various contexts. In the whole section, we stay in thr framework where $M:=L^2(\Omega)$, $E:=L^\infty(\Omega)$ and $V:=H^1_0(\Omega,\R^d)$.

 In subsection \ref{sub:honeycomb}, we exhibit a simple and efficient pair of approximation spaces $(M_h,V_h)$ called the honeycomb
 discretization pair, that numerically satisfies the discrete \emph{inf-sup} condition.

\subsection{Matrix formulation of the discretized problem}\label{sub:matrix}

In this section, we describe the matrix formulation of the discrete problem \eqref{eq:Th} which gives a way to use the stability theorems in practice. Let us fix a discretization size $h>0$ and pick a pair of finite dimensional subspaces $M_h\subset M$ and 
$V_h\subset V$. Let $(\e_1,\dots,\e_n)$ be a basis of $M_h$ and let $(\g e_1,\dots,\g e_p)$ be a basis of $V_h$. We define $\cT\in\R^{p\times n}$ and $\g b\in\R^{p}$ the 
matrix versions of the discrete operator $T_h$ and the right-hand side $\g f_h $as the matrices 
\begin{equation}\nonumber
\cT_{ij}:=\left<T_h\e_j\g e_i\right>_{V_h',V_h},\quad\text{and}\quad\g b_{i}:=\left<\g f_h,\g e_i\right>_{V_h',V_h}.
\end{equation}

As no ambiguity can occur, we adopt the notation for $\mu:=\sum_{j}\mu_j\e_j\in M_h$ and $\g \mu:=(\mu_1,\dots\mu_n)^T$ and the same notation 
for $\g v:=\sum_{i}v_i\g e_i\in V_h$ and ${\g v}=(v_1,\dots,v_p)^T\in\R^p$. We have the correspondence
\begin{equation}\nonumber
\left<T_h\mu,v\right>_{V_h',V_h}={\g v}^T\cT\g \mu.
\end{equation}
                                                                                                                                                                                                                                                                                                                                                                                                                                                                                                                                                                                                                                                                                                                                                                                                                   We now call $(\cS_{M})_{ij}:=\left<\e_i,\e_j\right>_M$ and $(\cS_{V})_{ij}:=\left<\g e_i,\g e_j\right>_V$. They enable to compute the norm in $M$ and $V$ through the formulas                                                                                                                                                                                                                                                                                                                                                                                                                                                                                                                                                                                                                                                                                                                                                                            $\norm{\mu}{M}^2=\sum_{i,j}\mu_i\mu_j\left<\e_i,\e_j\right>_M = \g \mu^T\cS_M\g \mu,$  and $ 
\norm{\g v}{V}^2=\sum_{i,j} v_i v_j\left<\g e_i, \g e_j\right>_V = {\g v}^T\cS_V{\g v}.$
If we denote $\cB_M$ and $\cB_V$ the square root matrices of $\cS_M$ and $\cS_V $ (i.e. such that $\cB_M^2=\cS_M$), we have that $\norm{\mu}{M}=\norm{\cB_M\g \mu}{2}$ and  $\norm{\g v}{V}=\norm{\cB_V{\g v}}{2}.
$ Hence the constant $\alpha(T_h)$ is given by 
\begin{equation}\begin{aligned}\nonumber
\alpha(T_h) &=\inf_{\g \mu\in\R^n} \sup_{\g v\in\R^p}\frac{\g v^T\cT\g \mu}{\norm{\cB_M\g \mu}{2}\norm{\cB_V\g v}{2}}\\
&=\inf_{\g \mu\in\R^n} \sup_{\g v\in\R^p}\frac{\g v^T\cB_V^{-1} \cT \cB_M^{-1}\g \mu}{\norm{\g \mu}{2}\norm{\g v}{2}} 
= \inf_{\g \mu\in\R^n} \frac{\norm{\cB_V^{-1} \cT \cB_M^{-1}\g \mu}{2}}{\norm{\g \mu}{2}}.
\end{aligned}\end{equation}
which is the smallest singular value of the matrix 
\begin{equation}\nonumber
\cM:=\cB_V^{-1} \cT \cB_M^{-1}
\end{equation}
or also the square root of the smallest eigenvalue of $\cM^T\cM=\cB_M^{-1}\cT^T\cS_V^{-1} \cT \cB_M^{-1}$.

Call now and $\g z\in \R^n$  the first singular vector of $\cM$ (hence associated with $\alpha(T_h)$) or the first eigenvector 
of $\cM^T\cM$. It is equal to the solution $z_h:=\sum_jz_{j}\e_j\in M_h$ of \eqref{eq:zh} up to a change of sign. 

\begin{remark} The basis matrices $\cB_M$ and $\cB_V$ are mandatory to get the exact solution $\alpha(T_h)$ and $z_h$ 
as defined in \eqref{eq:zh}. As $\alpha(T_h)$ is expected to be small, it is possible to consider directly the first singular
vector of the matrix $\cT$ itself. The numerical computation gets a bit simpler but creates an additional error which is not 
controlled by the theory described herein. 
\end{remark}

We can now compute the discrete \emph{inf-sup} constant of $T_h$:
\begin{equation}\label{eq:betah}\begin{aligned}
\beta(T_h) &=\inf_{\substack{ \g \mu\in\R^n\\ \g \mu \perp \cS_M\g z}} \sup_{\g v\in\R^p}\frac{\g v^T\cT\g \mu}{\norm{\cB_M\g \mu}{2}\norm{\cB_V\g v}{2}}\\
 &=\inf_{\substack{\g \mu \in\R^n\\ \g \mu \perp \g z}} \sup_{\g v\in\R^p}\frac{\g v^T(\cB_V^{-1})^T \cT \cB_M^{-1}\g \mu}{\norm{\g \mu}{2}\norm{\g v}{2}} = \inf_{\substack{\g \mu \in\R^n\\ \g \mu \perp \g z}} \frac{\norm{\cM\g \mu}{2}}{\norm{\g \mu}{2}}
\end{aligned}\end{equation}
which is the second smallest singular value of the matrix $\cM$ or also the square root of the second smallest eigenvalue
of $\cM^T\cM$. Finally, in order to give the solution of \eqref{eq:muh} in Theorem \ref{theo2}, we rewrite the problem under 
a matrix formulation:
\begin{equation}\nonumber
\begin{aligned}
\min_{\substack{ m\in M_h\\ m\perp z_h}} \norm{T_hm-\g f_h}{V_h'} &=\min_{\substack{ m\in M_h\\ m\perp z_h}}\sup_{\g v\in V_h} \frac{\inner{T_hm-\g f_h,\g v}{V_h',V_h}}{\norm{\g v}{V}}= \min_{\substack{ \g \mu\in \R^n\\ \g \mu\perp \g z}}\sup_{\g v\in \R^p} \frac{ \g v^T(\cT\g \mu - \g b)
}{\norm{\cB_V\g v}{2}}\\
&= \min_{\substack{ \g \mu\in \R^n\\ \g \mu\perp \g z}}\sup_{\g v\in \R^p} \frac{ \g v^T\cB_V^{-1}(\cT\g \mu - \g b)}{\norm{\g v}{2}}= \min_{\substack{ \g \mu\in \R^n\\ \g \mu\perp \g z}} \norm{\cB_V^{-1}(\cT\g \mu - \g b)}2.\\
\end{aligned}
\end{equation}

Call now $\wt\cT := \begin{bmatrix} \cT \\ \g z^T\end{bmatrix}$, $\wt{\g b}:=\left[\begin{matrix} \g b \\ 0 \end{matrix}\right]$ and $\wt\cB_V:=\left[\begin{matrix} \cB_{V} & 0 \\ 0 & 1 \end{matrix}\right]$ we aim at solving 
\begin{equation}\nonumber
\wt\cB_V^{-1}{\wt\cT}\g \mu = \wt\cB_V^{-1}\wt{\g b}
\end{equation}
in sense of least squares which is equivalent to define $\g \mu := ({\wt\cT}^T\wt\cS_V^{-1}{\wt\cT})^{-1}{\wt\cT}^T\wt\cS_V^{-1}\wt{\g b}.$

\subsection{The honeycomb pair of finite element spaces}\label{sub:honeycomb}

After numerous tests with various finite element pair of spaces, it appears that a specific pair of spaces gather a large amount of advantages for the specific use in the inverse parameter problem that we aim at solving. This pair $(M_h,V_h)$ is the so called honeycomb discretization pair. Like in Figure \ref{fig:honeycomb}, define a regular hexagonal subdivision of $\Omega$ denoted $\{\Omega_{h,j}^{\text{hex}}\}_{j=1,\dots,N^\text{hex}_h}$ where $h>0$ is the diameter of the hexagons and $N^\text{hex}_h$ is the number of hexagons used. We then call $\Omega_h\subset \Omega$ 
the subdomain defined by this subdivision. That means
\begin{equation}\nonumber
\overline{\Omega_h}=\bigcup_{j=1}^{N^\text{hex}_h}\overline{\Omega_{h,j}^{\text{hex}}}.
\end{equation}
Now we consider the uniform triangular sub-mesh defined by subdividing each hexagon in six equilateral triangles of size $h$. This subdivision is denoted $\{\Omega_{h,k}^{\text{tri}}\}_{k=1,\dots,N^\text{tri}_h}$ where $N^\text{tri}_h:=6N^\text{hex}_h$. It is represented in dashed bue in figure \ref{fig:honeycomb}. 

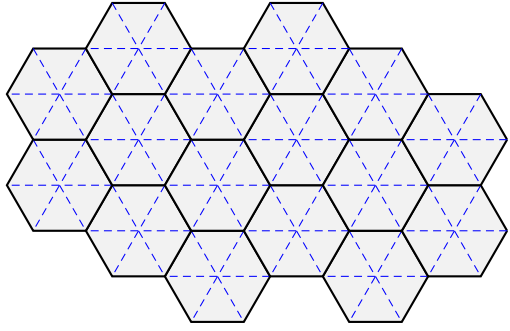
\begin{figure}
\begin{center}
\begin{tikzpicture}[scale = 0.7]
\newcommand{\hexagonA}{
\draw[fill=gray!10,thick]  (3*\k,{sqrt(3)*\l}) -- ++(120:1) -- ++(180:1) -- ++(240:1) --++(300:1) -- ++(0:1) --++ (60:1);
\foreach  \m in {0,...,2}{
\draw[blue,densely dashed]  (3*\k-1,{sqrt(3)*\l})++(60*\m:1)  -- ++(180+60*\m:2);
}
}
\newcommand{\hexagonB}{
\draw[fill=gray!10,thick]  ({3*\k+3/2},{sqrt(3)/2+sqrt(3)*\l}) -- ++(120:1) -- ++(180:1) -- ++(240:1) --++(300:1) -- ++(0:1) --++ (60:1);
\foreach  \m in {0,...,2}{
\draw[blue,densely dashed]   ({3*\k+3/2-1},{sqrt(3)/2+sqrt(3)*\l}) ++(60*\m:1)  -- ++(180+60*\m:2);
}
}
\foreach \k in {1,...,2}{\foreach \l in {0,...,0}{\hexagonA}}
\foreach \k in {0,...,2}{\foreach \l in {0,...,0}{\hexagonB}}
\foreach \k in {0,...,2}{\foreach \l in {1,...,1}{\hexagonA\hexagonB}}
\foreach \k in {0,...,2}{\foreach \l in {2,...,2}{\hexagonA}}
\foreach \k in {0,...,1}{\foreach \l in {2,...,2}{\hexagonB}}
\end{tikzpicture}

\caption{\label{fig:honeycomb} Honeycomb space discretization. In plain black, the hexagonal subdivision and in dashed blue,
the triangular subdivision.}
\end{center}
\end{figure}

We now define the finite dimensional discretization space $M_h$ of $M$ as the collection of functions $\mu\in L^2(\Omega_h)$ that are constant in each hexagon. In other terms,
\begin{equation}\nonumber
M_h:=\PP^0\left(\Omega_{h}^{\text{hex}}\right)=\left\{\mu\in L^2(\Omega_h)\ |\ \forall j\ \mu|_{\Omega_{h,j}^{\text{hex}}}\ \text{ is constant}\right\}.
\end{equation}
Functions in $M_h$ can be extended by $0$ out of $\Omega_h$ to get $M_h\subset M$. For the discretization space of $V$, we chose the classic finite element class $\PP^1_0$ over the triangulation. It is made of all the functions of $H^1_0(\Omega_h)$ that are linear over all the triangles. In other terms,
\begin{equation}\nonumber
V_h:=\PP_0^1\left(\Omega_{h}^{\text{tri}},\R^2\right)=\left\{\g v\in H^1_0(\Omega_h,\R^d)\ |\ \forall k\ \g v|_{\Omega_{h,k}^{\text{tri}}}\ \text{ is linear}\right\}.
\end{equation}
Functions in $V_h$ can be extended by $\g 0$ out of $\Omega_h$ to get $V_h\subset V$.

\begin{remark} This particular choice of finite element spaces gathers several advantages to compare to other more classic pairs:

\begin{enumerate}
\item The space $\PP^0\left(\Omega_{h}^{\text{hex}}\right)$ is suitable for discontinuous functions interpolation. This is important as we aim at recovering discontinuous mechanical parameters of biological tissues for instance. 

\item The hexagonal discretization of $\Omega$ is optimal in the sense that it minimizes the ratio of the number of unknown
$N_h^{\text{hex}}$ over the resolution $h$. 

\item From a given hexagonal mesh and triangular sub-mesh, spaces $\PP^0\left(\Omega_{h}^{\text{hex}}\right)$ and 
$\PP_0^1\left(\Omega_{h}^{\text{tri}},\R^2\right)$ are easy to build from the most classic pair of finite element
spaces $\left(\PP^0\left(\Omega_{h}^{\text{tri}}\right), \PP^1\left(\Omega_{h}^{\text{tri}}\right)\right)$.

\item The system of equations $T_h\mu_h=\g f_h$ is (most of the time) over-determinate as it involves around $2 N_h^{\text{hex}}$ 
equations for $N_h^{\text{hex}}$ unknown. Note that as we solve the problem in the sense of least squares, over-determination is
not a problem while under-determination is. 

\item This pair gives an excellent evaluation of the discrete $\emph{inf-sup}$ constant $\beta(T_h)$ that is the key element for 
discrete stability. 
\end{enumerate}

\end{remark}

\subsection{Inverse gradient problem}

Let $\Omega$ be the unit square $(0,1)^2.$ We approach here the solution $\mu\in L^\infty(\Omega)$ of the problem 
$-\nabla \mu = \g f$ where $\g f$ is given vectorial function. This case correspond to \eqref{eq:1} where $S=I$ everywhere.
In this case, many simplification occur as $T_h:=-\nabla|_{M_h}$ and then $\e_h^{\text{op}}=0$. 
Moreover $\rho(T)\leq 1$. In the absence of noise, the result of Theorem \ref{theo2} reads : $
\frac{\norm{\mu_{h}-\pi_h\mu}{M}}{\norm{\pi_h\mu}{M}} \leq \frac{4}{\beta\left(T_h\right)}\left(\e^{\text{rhs}}_h+\e^{\text{int}}_h(\mu)\right)$
where $\mu_h$ is the solution of $\min_{\mu\in M_h}\norm{T_h\mu -\g f}{V_h'}$ under the condition 
$\mu_h\in L^2_0(\Omega_h)$ i.e. $\int_{\Omega_h}\mu_h =0$.

Let first compute $\beta(T_h)$ using \eqref{eq:betah} at check its behavior when $h$ got to $0$. In figure \ref{fig:betah} we see
that it seem to converge to some $\beta_0>0$ lower than the conjectured \emph{inf-sup} constant $\beta(\nabla)=\sqrt{1/2-1/\pi}$ in 
the unit square (see \cite[Theorem 3.3]{costabel2015inf} for details about this conjectured value).

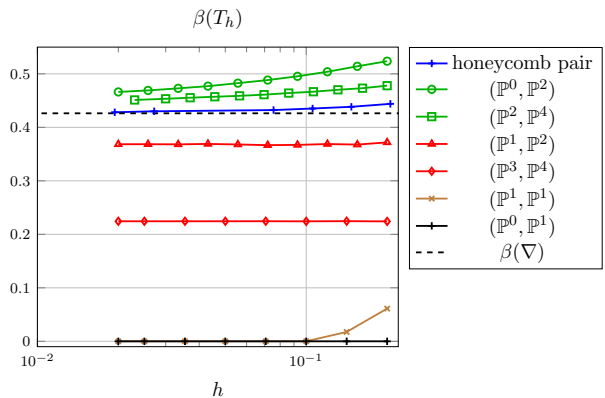
\begin{figure}\begin{center}
\begin{tikzpicture}[scale=0.7]
\begin{semilogxaxis}[
xmin = 0.01, xmax = 0.22, ymin = -0.01, ymax = 0.55,
grid = major,
xlabel={$h$},
 legend entries = {honeycomb pair,$(\PP^0{,\, }\PP^2)$,$(\PP^2{,\, }\PP^4)$,$(\PP^1{,\, }\PP^2)$,$(\PP^3{,\, }\PP^4)$,$(\PP^1{,\, }\PP^1)$,$(\PP^0{,\, }\PP^1)$,$\beta(\nabla)$,},
 legend pos=outer north east,
 title = $\beta(T_h)$
 ]
 
\addplot [ line width = 1 pt,blue, mark=+] table {
    0.0194    0.4282
    0.0272    0.4300
    0.0756    0.4323
    0.1056    0.4352
    0.1472    0.4384
    0.2056    0.4439
};

\addplot [line width = 1 pt,black!30!green,mark=o] table {
0.0200    0.4662
0.0258    0.4690
0.0334    0.4729
0.0431    0.4770
0.0557    0.4826
0.0719    0.4884
0.0928    0.4954
0.1199    0.5039
0.1549    0.5141
0.2000    0.5237
};

\addplot [ line width = 1 pt,black!30!green,mark=square] table {
0.0230    0.4512
0.0300    0.4534
0.0370    0.4554
0.0457    0.4572
0.0565    0.4589
0.0697    0.4611
0.0861    0.4641
0.1063    0.4666
0.1312    0.4702
0.1620    0.4732
0.2000    0.4779
};

\addplot [ line width = 1 pt,red,mark=triangle] table {
0.0200    0.3684
    0.0258    0.3685
    0.0334    0.3682
    0.0431    0.3690
    0.0557    0.3680
    0.0719    0.3667
    0.0928    0.3672
    0.1199    0.3688
    0.1549    0.3678
   0.2000    0.3719
};

\addplot [ line width = 1 pt,red,mark=diamond] table {
0.0200    0.2244
0.0250    0.2244
0.0354    0.2244
0.0500    0.2244
0.0707    0.2245
0.1000    0.2243
0.1414    0.2246
0.2000    0.2242
};

\addplot [ line width = 1 pt, brown,mark=x] table {
0.0200    0.0000
0.0250    0.0000
0.0354    0.0000
0.0500    0.0000
0.0707    0.0000
 0.1000    0.0001
0.1414    0.0175
0.2000    0.0612
};

\addplot [ line width = 1 pt, black,mark=+] table {
0.0200    0.0000
0.0250    0.0000
0.0354    0.0000
0.0500    0.0000
0.0707    0.0000
 0.1000    0.000
0.1414    0.0000
0.2000    0.0000
};

\addplot [domain=0.004:0.3,  black, line width = 1 pt, dashed] {0.42625};

\end{semilogxaxis}
\end{tikzpicture}
\caption{\label{fig:betah}Behavior of the discrete \emph{inf-sup} constant $\beta(T_h)$ for the inverse gradient problem in the unit square $\Omega:=(0,1)^2$, for various choices of pair of discretization spaces. The dashed line represents the conjectured value of the $\emph{inf-sup}$ constant $\beta(\nabla)=\sqrt{1/2-1/\pi}$ of the gradient operator in $\Omega$.} 
\end{center}\end{figure}

Consider now a smooth map $\mu_1(x):=\cos(10x_1)+\cos(10x_2)$ for $x\in\Omega$, for such a smooth function we expect an error of interpolation in $M_h$ of order $\e_h^{\text{int}}(\mu_1)=\cO(h)$ and an error of interpolation of its gradient on $V_h'$ of order $\e_h^{\text{rhs}}=\cO(h^2)$. Hence the relative error $E_1(h):={\norm{\mu_{1,h}-\pi_h\mu_1}{M}}/{\norm{\pi_h\mu_1}{M}}$ is expected to be at least of order $\cO(h)$. In figure \ref{fig:E1} we observe a convergence of order 2 in absence of noise. We retry the same test with piecewise constant $\mu_2$. Its derivative is approached first in $\PP^0(\Omega^{\text{tri}}_h)$ 
to deduce its vectorial form in $V_h'$. We observe a convergence of order $1/2$ in absence of noise.

To illustrate the stability with respect to noise on the right-had side, we corrupt the data $-\nabla \mu$ with 
the multiplication term-by-term by $1+\sigma\mathcal{N}$ where $\sigma>0$ is the noise level and $\mathcal{N}$ is a Gaussian random 
variable of variance one. 

\begin{figure}\begin{center}
\def\textscale{0.7}
\def\imagescale{0.3272}
\def\scale{\textscale}
\def\witdh{\imagescale/\textscale*5 cm}
\def\height{\imagescale/\textscale*1*5 cm}
\def\xmin{0}\def\xmax{1}\def\ymin{0}\def\ymax{1}

\begin{tikzpicture}[scale=\scale]
\begin{axis}[width=\witdh, height=\height, axis on top, scale only axis, xmin=\xmin, xmax=\xmax, ymin=\ymin, ymax=\ymax, colormap/jet, colorbar, point meta min=-2, point meta max=2]
\addplot graphics [xmin=\xmin,xmax=\xmax,ymin=\ymin,ymax=\ymax]{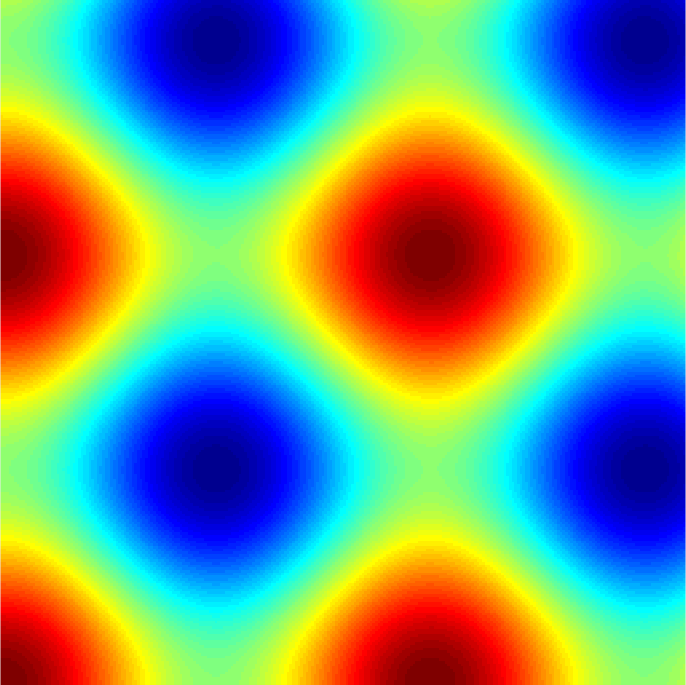};
\end{axis}
\end{tikzpicture}
\begin{tikzpicture}[scale=\scale]
\begin{axis}[width=\witdh, height=\height, axis on top, scale only axis, xmin=\xmin, xmax=\xmax, ymin=\ymin, ymax=\ymax, colormap/jet, colorbar,point meta min=-2,point meta max=2]
\addplot graphics [xmin=\xmin,xmax=\xmax,ymin=\ymin,ymax=\ymax]{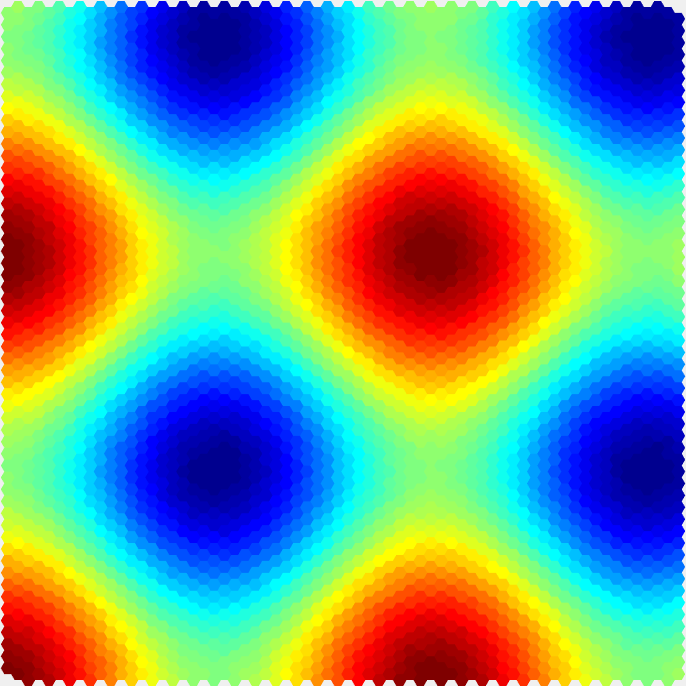};
\end{axis}
\end{tikzpicture}
\begin{tikzpicture}[scale=\scale]
\begin{axis}[width=\witdh, height=\height, axis on top, scale only axis, xmin=\xmin, xmax=\xmax, ymin=\ymin, ymax=\ymax, colormap/jet, colorbar,point meta min=-2,point meta max=2]
\addplot graphics [xmin=\xmin,xmax=\xmax,ymin=\ymin,ymax=\ymax]{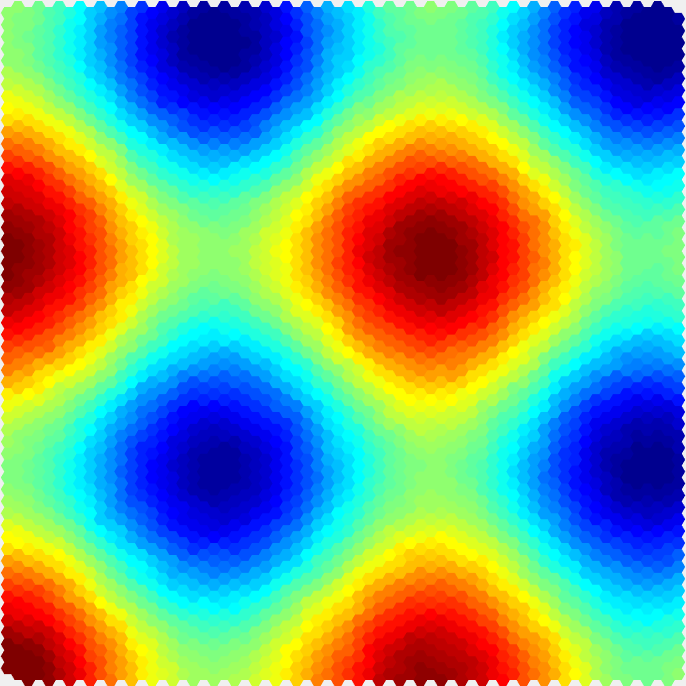};
\end{axis}
\end{tikzpicture}
\begin{tikzpicture}[scale=\scale]
\begin{axis}[width=\witdh, height=\height, axis on top, scale only axis, xmin=\xmin, xmax=\xmax, ymin=\ymin, ymax=\ymax, colormap/jet, colorbar,point meta min=-2,point meta max=2]
\addplot graphics [xmin=\xmin,xmax=\xmax,ymin=\ymin,ymax=\ymax]{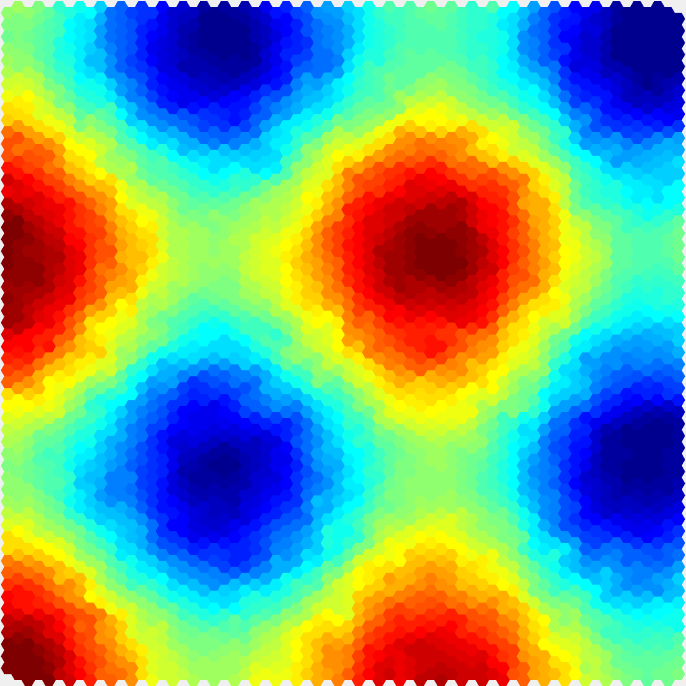};
\end{axis}
\end{tikzpicture}

\begin{tikzpicture}[scale=\scale]
\begin{axis}[width=\witdh, height=\height, axis on top, scale only axis, xmin=\xmin, xmax=\xmax, ymin=\ymin, ymax=\ymax, colormap/jet, colorbar,point meta min=-1.0735,point meta max=1.9265]
\addplot graphics [xmin=\xmin,xmax=\xmax,ymin=\ymin,ymax=\ymax]{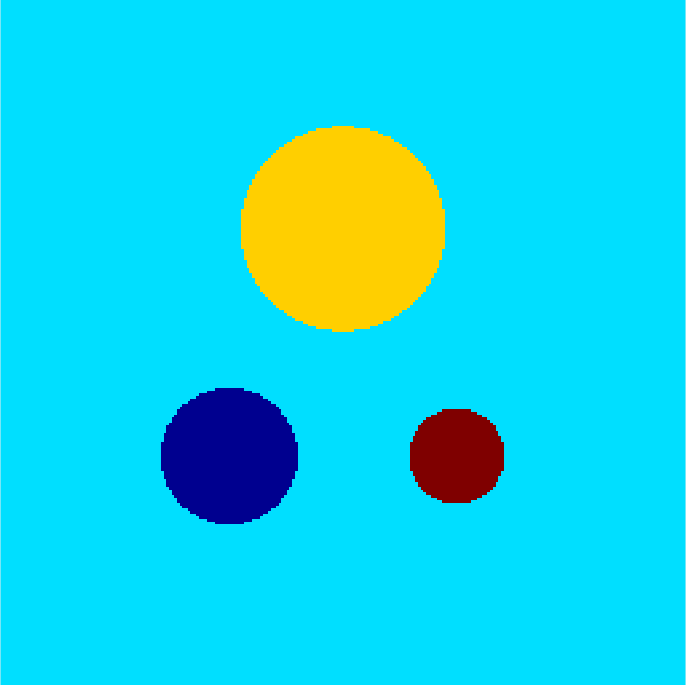};
\end{axis}
\end{tikzpicture}
\begin{tikzpicture}[scale=\scale]
\begin{axis}[width=\witdh, height=\height, axis on top, scale only axis, xmin=\xmin, xmax=\xmax, ymin=\ymin, ymax=\ymax, colormap/jet, colorbar,point meta min=-1.0735,point meta max=1.9265]
\addplot graphics [xmin=\xmin,xmax=\xmax,ymin=\ymin,ymax=\ymax]{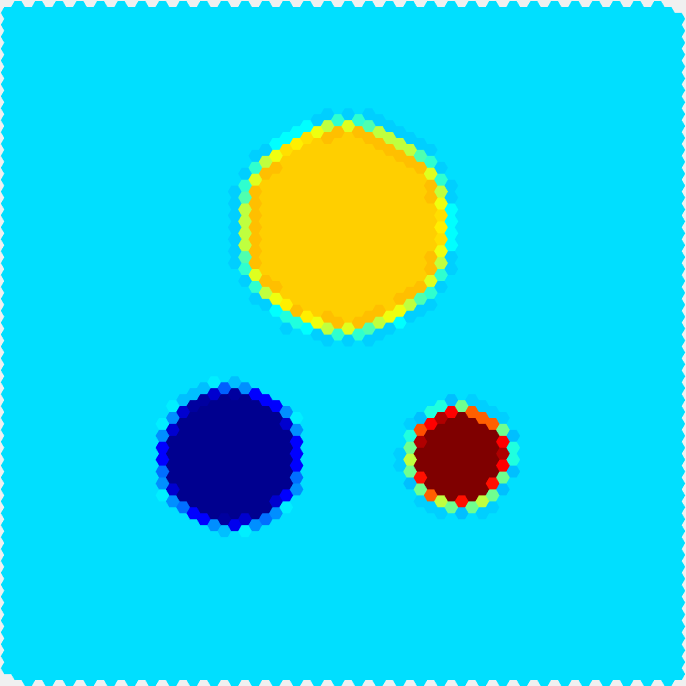};
\end{axis}
\end{tikzpicture}
\begin{tikzpicture}[scale=\scale]
\begin{axis}[width=\witdh, height=\height, axis on top, scale only axis, xmin=\xmin, xmax=\xmax, ymin=\ymin, ymax=\ymax, colormap/jet, colorbar,point meta min=-1.0735,point meta max=1.9265]
\addplot graphics [xmin=\xmin,xmax=\xmax,ymin=\ymin,ymax=\ymax]{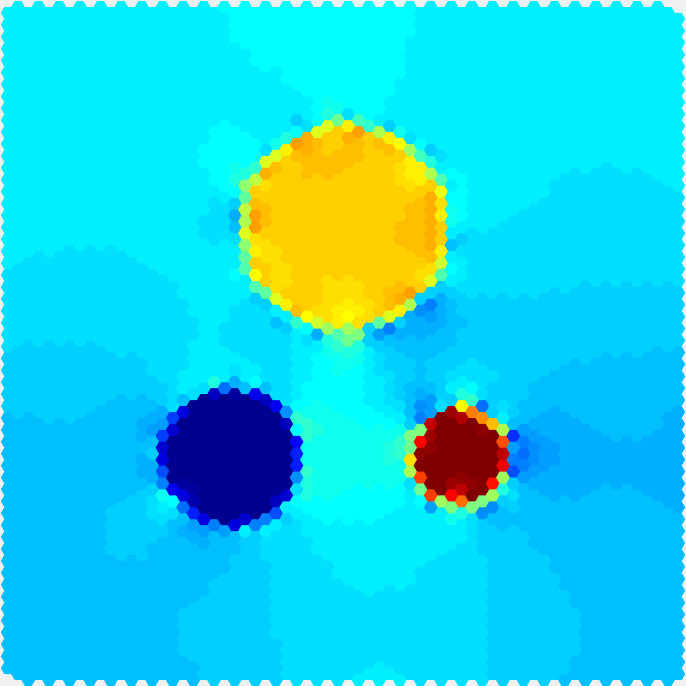};
\end{axis}
\end{tikzpicture}
\begin{tikzpicture}[scale=\scale]
\begin{axis}[width=\witdh, height=\height, axis on top, scale only axis, xmin=\xmin, xmax=\xmax, ymin=\ymin, ymax=\ymax, colormap/jet, colorbar,point meta min=-1.0735,point meta max=1.9265]
\addplot graphics [xmin=\xmin,xmax=\xmax,ymin=\ymin,ymax=\ymax]{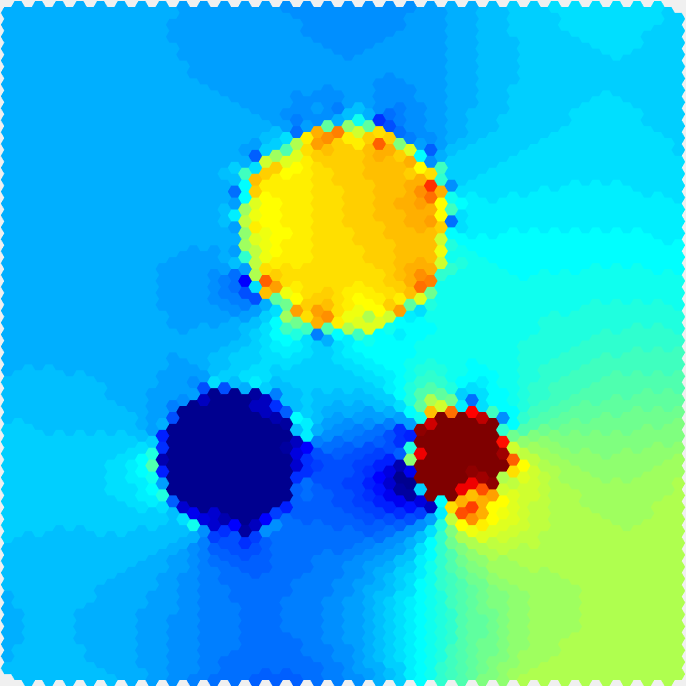};
\end{axis}
\end{tikzpicture}
\caption{\label{fig:reco1} Numerical stability of the reconstruction of maps $\mu_1$ and $\mu_2$  using method given by Theorem \ref{theo2} with resolution $h=0.01$. From left to right: column 1: exact map to recover, 2. reconstruction with no noise, column 3: reconstruction with noise level $\sigma = 1$, column 4: reconstruction with noise level $\sigma = 2$.}
\end{center}\end{figure}

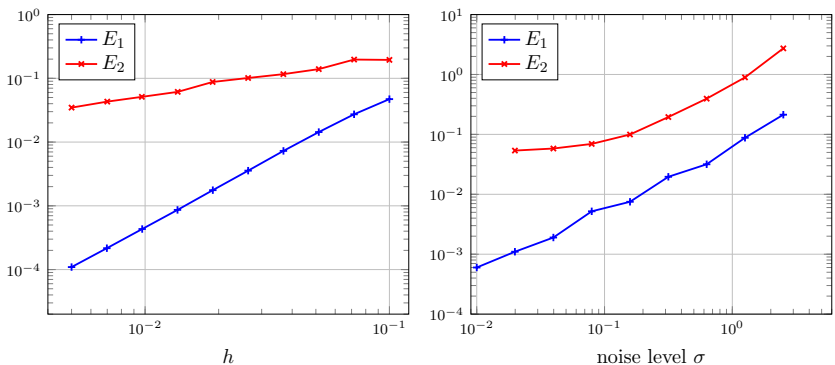
\begin{figure}\begin{center}
\begin{tikzpicture}[scale=0.7]
\begin{loglogaxis}[xmin = 0.004, xmax = 0.12, ymin = 2e-5, ymax = 1, grid = major, xlabel={$h$}, legend entries = {$E_1$,$E_2$}, legend pos = north west]
\addplot[ line width = 1 pt, blue,mark=+] table {
0.0050000   0.0001094
0.0069747   0.0002164
0.0097294   0.0004325
0.0135720   0.0008626
0.0189323   0.0017608
0.0264097   0.0035608
0.0368403   0.0072698
0.0513904   0.0143510
0.0716871   0.0272670
0.1000000   0.0471591
};
\addplot[ line width = 1 pt, red,mark=x] table {
0.0050    0.0347
0.0070    0.0431
0.0097    0.0512
0.0136    0.0613
0.0189    0.0876
0.0264    0.1012
0.0368    0.1162
0.0514    0.1394
0.0717    0.1969
0.1000    0.1945
};

\end{loglogaxis}
\end{tikzpicture}
\begin{tikzpicture}[scale=0.7]
\begin{loglogaxis}[
xmin = 0.009, xmax =6, ymin = 1e-4, ymax = 10, grid = major, xlabel={noise level $\sigma$}, legend entries = {$E_1$,$E_2$}, legend pos = north west,]
\addplot[ line width = 1 pt, blue,mark=+] table {
0.0100    0.0006
0.0199    0.0011
0.0398    0.0019
0.0794    0.0052
0.1583    0.0075
0.3158    0.0196
0.6300    0.0317
1.2566    0.0878
2.5066    0.2131
};

\addplot[line width = 1 pt, red,mark=x] table {
0.0199    0.0537
0.0398    0.0581
0.0794    0.0693
0.1583    0.0996
0.3158    0.1948
0.6300    0.3949
1.2566    0.8937
2.5066    2.7282
};

\end{loglogaxis}
\end{tikzpicture}
\caption{\label{fig:E1} Left : relative $L^2$-error on the reconstruction with respect to $h$ in the absence of noise. Right : relative $L^2$-error on the reconstruction with respect to the noise level $\sigma$ with $h=0.01$.} 
\end{center}\end{figure}

\subsection{Quasi-static elastography}

\paragraph{Forward problem} To illustrate the ability of solving a quasi-static elastography problem 
in the case $\lambda=0$ from a single measurement, we compute a virtual data field by solving the linear elastic forward problem 
\begin{equation}\label{eq:forward}\left\{\begin{aligned}
-\nabla\cdot(2\mu_{\text{exact}}\, \cE(\g u)) &= \g 0\quad\tin (0,1)^2,\\
2\mu_{\text{exact}}\, \cE(\g u)\cdot \g \nu &= \g f\quad\ton (0,1)\times\{1\},\\
\cE(\g u)\cdot \g \nu &= \g 0\quad\ton (0,1)\times\{0\},\\
\g u &= \g 0,\quad\ton \{0,1\}\times (0,1).\\
\end{aligned}\right.\end{equation}
where $\mu_{\text{exact}}$ is described in Figure \ref{fig:staticdata}. We chose here a constant 
boundary force $\g f:=(1,-1)^T$. This problem is solved using classic $\PP^1$ finite element method over an unstructured 
triangular mesh. The computed data field $\g u$ is then stored in a cartesian grid to avoid any numerical inverse crime. 
It is represented in Figure \ref{fig:staticdata}.\\

\begin{figure}\begin{center}
\def\textscale{0.7}
\def\imagescale{0.5}
\def\scale{\textscale}
\def\witdh{\imagescale/\textscale*5 cm}
\def\height{\imagescale/\textscale*1*5 cm}
\def\xmin{0}\def\xmax{1}\def\ymin{0}\def\ymax{1}

\def\scale{\textscale}
\def\witdh{\imagescale/\textscale*5 cm}
\def\height{\imagescale/\textscale*1*5 cm}
\def\xmin{0}\def\xmax{1}\def\ymin{0}\def\ymax{1}

\begin{tikzpicture}[scale=\scale]
\begin{axis}[title = $\mu_{\text{exact}}$, width=\witdh, height=\height, axis on top, scale only axis, xmin=\xmin, xmax=\xmax, ymin=\ymin, ymax=\ymax, colormap/jet, colorbar,point meta min=0.5,point meta max=2]
\addplot graphics [xmin=\xmin,xmax=\xmax,ymin=\ymin,ymax=\ymax]{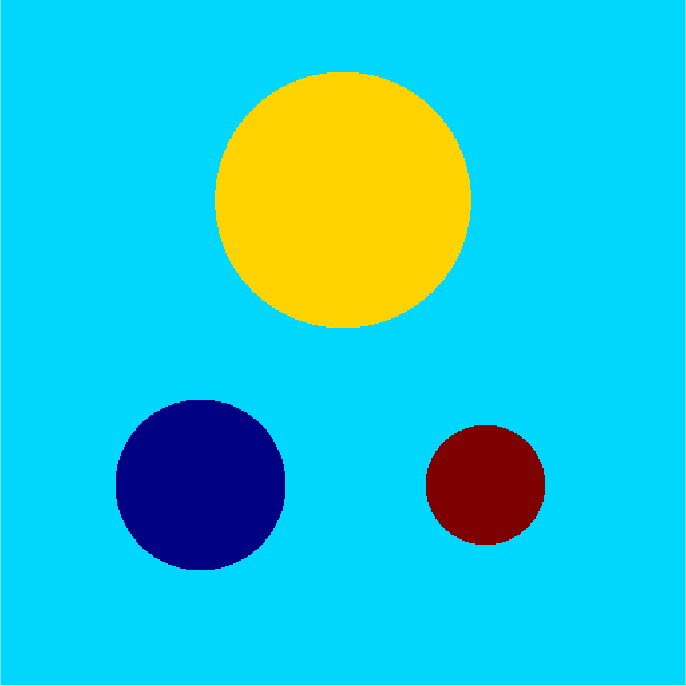};
\end{axis}
\end{tikzpicture}
\begin{tikzpicture}[scale=\scale]
\begin{axis}[title = $u_1$, width=\witdh, height=\height, axis on top, scale only axis, xmin=\xmin, xmax=\xmax, ymin=\ymin, ymax=\ymax, colormap/jet, colorbar,point meta min=-1.6793,point meta max=1.6793]
\addplot graphics [xmin=\xmin,xmax=\xmax,ymin=\ymin,ymax=\ymax]{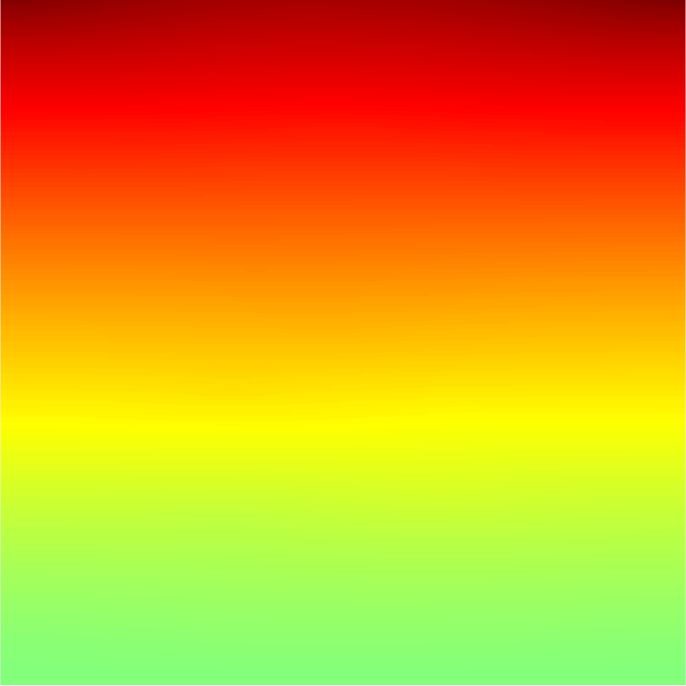};
\end{axis}
\end{tikzpicture}
\begin{tikzpicture}[scale=\scale]
\begin{axis}[title = $u_2$,width=\witdh, height=\height, axis on top, scale only axis, xmin=\xmin, xmax=\xmax, ymin=\ymin, ymax=\ymax, colormap/jet, colorbar,point meta min=-1.6793,point meta max=1.6793]
\addplot graphics [xmin=\xmin,xmax=\xmax,ymin=\ymin,ymax=\ymax]{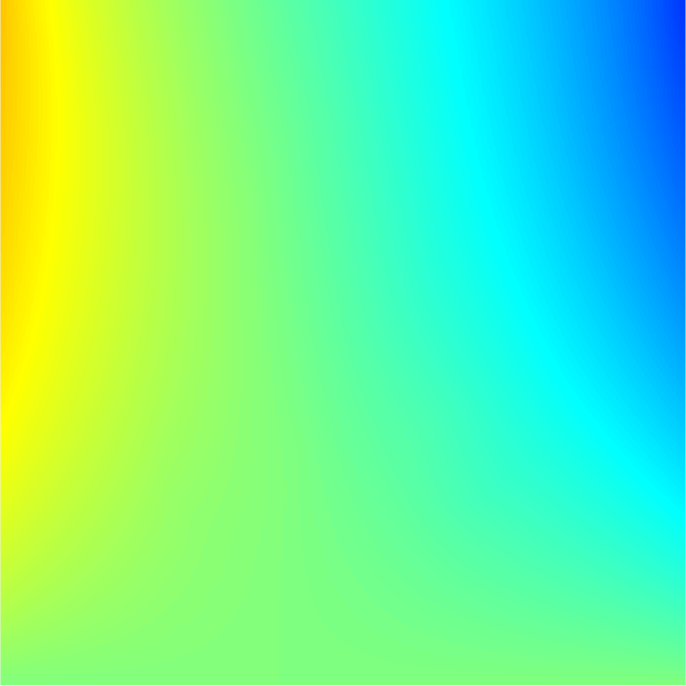};
\end{axis}
\end{tikzpicture}

\caption{\label{fig:staticdata} First line, from left to right: The exact map $\mu_\text{exact}$, the two components of 
the data field $\g u=(u_1,u_2)$ computed via \eqref{eq:forward}, the only data used to inverse the problem.}
\end{center}\end{figure}
\paragraph{Inverse problem} From this data, we approach 
the matrix $S:=2\cE(\g u)$ through an exact differentiation of $\g u$ on the finite element space.
We then chose a particular pair of spaces $(M_h,V_h)$ suitable 
for the inverse parameter problem and we define the matrix form of the approached operator $T_h$.
Before applying Theorem \ref{theo1} we compute the discrete values of $\alpha(T_h)$ and
$\beta(T_h)$ for few pairs of spaces (see Figure \ref{fig:static_betah}). We here control that $\beta(T_h)$ does not 
vanish and that the ratio $\alpha(T_h)/\beta(T_h)$ is small enough. We recall that this is needed for good error estimates using Theorem \ref{theo1}. 
Note that the honeycomb pair shows a much better behavior than the other consider pairs of spaces. \\

We plot now solutions $\mu_h$ of the numerical inversion with various choice of pair of spaces 
in Figure \ref{fig:resultelasto}.  Then in Figure \ref{fig:table} we present tables of comparisons of
different pair of spaces in terms of relative error and complexity through the number of degrees 
of freedom and number of equations. In particular,
\begin{itemize}
 \item As expected and for all choice of pair of spaces satisfying inf-sup condition, 
 the numerical approximation  ${\bf u}$ gives some nice reconstruction of  the elastic 
 coefficient $2 \mu_{exact}$. Moreover, in each case, we also clearly observe  a quantitative convergence as 
 $h \to 0$.
 \item The numerical solutions obtained with the honeycomb approach  give some better reconstruction 
 than using other pair of spaces. It can be explained by a better ratio $\alpha(T_h)/\beta(T_h)$.
 \item The use of high degree as with the pair of spaces $(\PP^4{,}\PP^2)$
 raises some  numerical memory issues in the computation the matrix  
 $\cB_M^{-1}$ and $\cS_V^{-1}$. In particular, we don't succeed to reach time steps $h$ smaller than
 $h = 0.025$ with a standard laptop. 
  \item From a computation cost point of view, the honeycomb approach has also many advantages. 
 The matrix $\cS_M$ and $\cS_V$ are respectively  diagonal and tri-diagonal which 
 greatly facilitate the computation of $\cB_M^{-1} = \sqrt{\cS_M^{-1}}$ and $\cS_V^{-1}$.
 Finally, we can reach much finer resolutions than using  other finite element space proposed in this paper.
\end{itemize}

\def\scale{0.49}

\begin{figure}\begin{center}
\begin{tikzpicture}[scale=\scale]
\begin{loglogaxis}[
xmin = 0.01, xmax = 0.2, ymin = 1e-6, ymax = 1,
grid = major,
xlabel={$h$},
 legend pos=outer north east,
 title = $\alpha(T_h)$
 ]
 
\addplot [ line width = 2 pt,blue, mark=+] table {
0.020    0.0028
0.026    0.0035
0.037    0.0044
0.05    0.0044
0.07    0.0067
0.10    0.0091
0.14    0.0100
};

\addplot [line width = 1 pt,black!30!green,mark=o] table {
   0.140000000000000   0.155363245837991
   0.100000000000000   0.109304710221934
   0.070000000000000   0.079303727641809
   0.050000000000000   0.070727838058673
   0.037000000000000   0.058152609779338
   0.025000000000000   0.05495301874895
};

\addplot [ line width = 1 pt,black!30!green,mark=square] table {
0.1400    0.1222
0.1000    0.0906
0.0700    0.0803
0.0500    0.0647
0.0370    0.0490
0.0250    0.0437
0.0200      0.039
};

\addplot [ line width = 1 pt,red,mark=triangle] table {
0.140000000000000   0.104681158114168
   0.100000000000000   0.099847657411454
   0.070000000000000   0.086941661566795
   0.050000000000000   0.071026031445884
   0.037000000000000   0.063004885377747
   0.025000000000000   0.050633497357910
   0.020000000000000   0.040419122053869

};

\addplot [ line width = 1 pt,red,mark=diamond] table {
0.1400    0.0901
    0.1000    0.0628
    0.0700    0.0499
    0.0500    0.0431
    0.0370    0.0328
    0.0250    0.0269
    0.02      0.022
};

\addplot [ line width = 1 pt, brown,mark=x] table {
0.140000000000000   0.032627171888366
   0.100000000000000   0.024227784245503
   0.070000000000000   0.000014764730127
   0.050000000000000   0.000006309867297
   0.037000000000000   0.000003963102423
};

\end{loglogaxis}
\end{tikzpicture}
\begin{tikzpicture}[scale=\scale]
\begin{loglogaxis}[
xmin = 0.01, xmax = 0.2, ymin = 1e-6, ymax = 1,
grid = major,
xlabel={$h$},
 legend pos=outer north east,
 title = $\beta(T_h)$
 ]
 
\addplot [ line width = 1 pt,blue, mark=+] table {
0.020    0.0417
0.026    0.0449
0.034    0.046
0.05    0.0490
0.069    0.0589
0.10    0.0704
0.14    0.0672
};

\addplot [line width = 1 pt,black!30!green,mark=o] table {
   0.140000000000000   0.314144326441098
   0.100000000000000   0.269616915589222
   0.070000000000000   0.209782255618618
   0.050000000000000   0.157214484915668
   0.037000000000000   0.146348986206567
   0.025000000000000   0.133214036416270
   0.020000000000000   0.123676345435032
  0.015000000000000   0.114743023061429

};

\addplot [ line width = 1 pt,black!30!green,mark=square] table {
0.1400    0.1970
    0.1000    0.1845
    0.0700    0.1689
    0.0500    0.1562
    0.0370    0.1353
    0.0250    0.1150
    0.0200    0.105
};

\addplot [ line width = 1 pt,red,mark=triangle] table {
0.140000000000000   0.211910481040829
   0.100000000000000   0.195619138908986
   0.070000000000000   0.165986325019177
   0.050000000000000   0.168073473359580
   0.037000000000000   0.144946446010525
   0.025000000000000   0.133832422250369
   0.020000000000000   0.125446077545610
};

\addplot [ line width = 1 pt,red,mark=diamond] table {
0.1400    0.1060
    0.1000    0.1203
    0.0700    0.0935
    0.0500    0.0930
    0.0370    0.0879
    0.0250    0.0806
    0.02       0.078
};

\addplot [ line width = 1 pt, brown,mark=x] table {
0.140000000000000   0.069821502297676
   0.100000000000000   0.037692188614761
   0.070000000000000   0.000018543203196
   0.050000000000000   0.000010918532134
   0.037000000000000   0.000013290103057

};

\end{loglogaxis}
\end{tikzpicture}
\begin{tikzpicture}[scale=\scale]
\begin{loglogaxis}[
xmin = 0.01, xmax = 0.2, ymin = 1e-2, ymax = 1,
grid = major,
xlabel={$h$},
 legend entries = {honeycomb,$(\PP^0{,\, }\PP^2)$,$(\PP^2{,\, }\PP^4)$,$(\PP^1{,\, }\PP^2)$,$(\PP^3{,\, }\PP^4)$,$(\PP^1{,\, }\PP^1)$,$(\PP^0{,\, }\PP^1)$},
 legend pos=outer north east,
 title = $\alpha(T_h)/\beta(T_h)$
 ]
 
\addplot [ line width = 1 pt,blue, mark=+] table {
    0.0200    0.0513
    0.0280    0.0644
    0.0400    0.0671
    0.0520    0.0780
    0.0680    0.0957
    0.1000    0.0898
    0.1380    0.1138
};

\addplot [line width = 1 pt,black!30!green,mark=o] table {
0.140000000000000   0.494560088345641
   0.100000000000000   0.405407464821188
   0.070000000000000 0.378028768009734  
   0.050000000000000  0.449881180456194 
   0.037000000000000  0.397355740457658 
   0.025000000000000  0.412516730423488 
   0.020000000000000   0.378917033253727 
};

\addplot [ line width = 1 pt,black!30!green,mark=square] table {
    0.1400    0.6201
    0.1000    0.4912
    0.0700    0.4757
    0.0500    0.4141
    0.0370    0.3622
    0.0250    0.3794
    0.02       0.3714
};

\addplot [ line width = 1 pt,red,mark=triangle] table {
0.140000000000000   0.493987638553842
   0.100000000000000   0.510418653145740
   0.070000000000000   0.523788098548182
   0.050000000000000   0.422589180946652
   0.037000000000000   0.434676993550929
   0.025000000000000   0.378335058923067
   0.020000000000000   0.322203155687938

};

\addplot [ line width = 1 pt,red,mark=diamond] table {
0.1400    0.8502
    0.1000    0.5223
    0.0700    0.5338
    0.0500    0.4634
    0.0370    0.3729
    0.0250    0.3338
    0.02       0.2821
};

\addplot [ line width = 1 pt, brown,mark=x] table {
0.140000000000000   0.467294040011682
   0.100000000000000   0.642779980041143
   0.070000000000000   0.796234068654477
   0.050000000000000   0.577904357456867
   0.037000000000000   0.298199525344108
   0.025000000000000   0.776606893060210
};
\end{loglogaxis}
\end{tikzpicture}
\caption{\label{fig:static_betah}Behavior of the contants $\alpha(T_h)$, $\beta(T_h)$ and the ratio $\alpha(T_h)/\beta(T_h)$ for the inverse static elastography problem in the unit square $\Omega:=(0,1)^2$, for various choices of pair of discretization spaces.} 
\end{center}\end{figure}
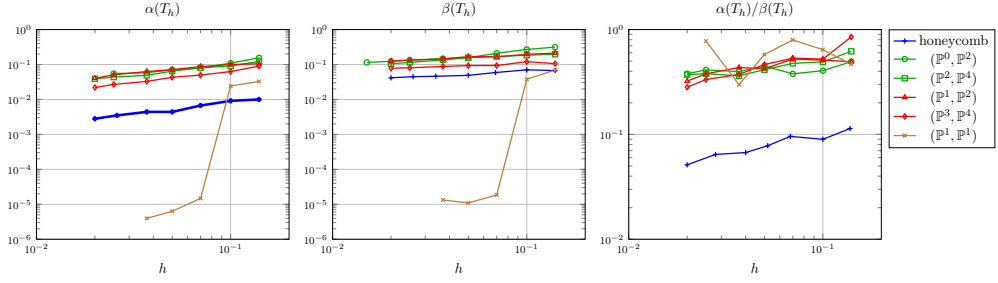
%


\begin{figure}\begin{center}
\def\textscale{0.7}
\def\imagescale{0.54}
\def\scale{\textscale}
\def\witdh{\imagescale/\textscale*5 cm}
\def\height{\imagescale/\textscale*1*5 cm}
\def\xmin{0}\def\xmax{1}\def\ymin{0}\def\ymax{1}

\def\xmin{0.05}\def\xmax{0.95}\def\ymin{0.05}\def\ymax{0.95}

\hspace{0.2cm}
 \begin{tikzpicture}[scale=\scale]
\begin{axis}[width=\witdh, height=\height, axis on top, scale only axis, xmin=\xmin, xmax=\xmax, ymin=\ymin, ymax=\ymax, colormap/jet, colorbar,point meta min=1,point meta max=4,ylabel=honeycomb,title={$h=0.05$}]
\addplot graphics [xmin=\xmin,xmax=\xmax,ymin=\ymin,ymax=\ymax]{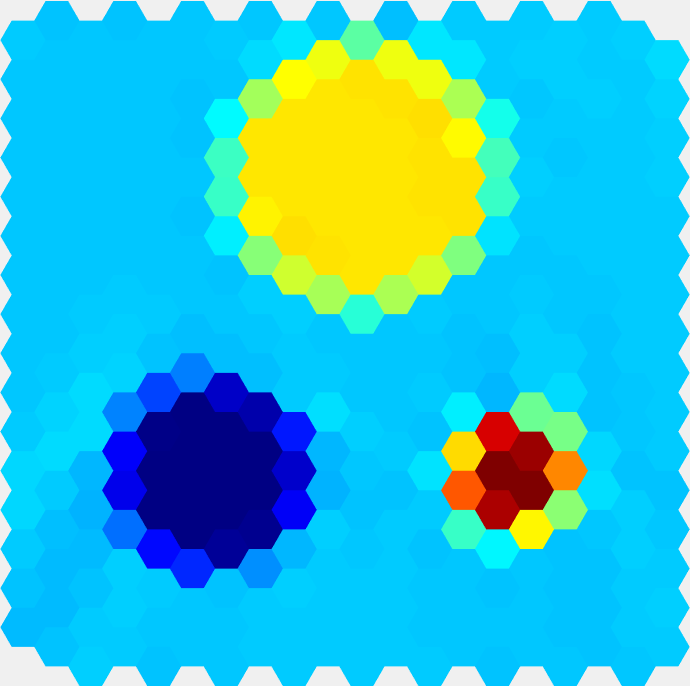};
\end{axis}
\end{tikzpicture}
 \begin{tikzpicture}[scale=\scale]
\begin{axis}[width=\witdh, height=\height, axis on top, scale only axis, xmin=\xmin, xmax=\xmax, ymin=\ymin, ymax=\ymax, colormap/jet, colorbar,point meta min=1,point meta max=4,title={$h=0.025$}]
\addplot graphics [xmin=\xmin,xmax=\xmax,ymin=\ymin,ymax=\ymax]{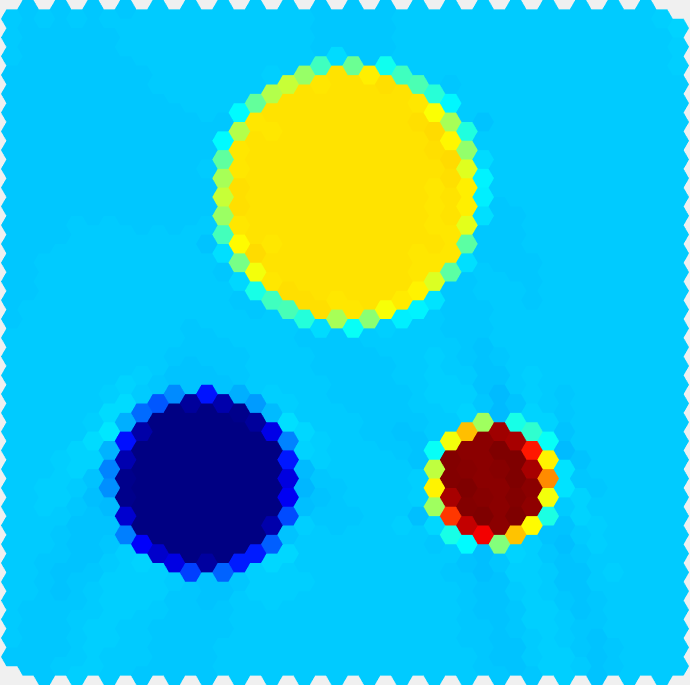};
\end{axis}
\end{tikzpicture}
 \begin{tikzpicture}[scale=\scale]
\begin{axis}[width=\witdh, height=\height, axis on top, scale only axis, xmin=\xmin, xmax=\xmax, ymin=\ymin, ymax=\ymax, colormap/jet, colorbar,point meta min=1,point meta max=4,title={$h=0.01$}]
\addplot graphics [xmin=\xmin,xmax=\xmax,ymin=\ymin,ymax=\ymax]{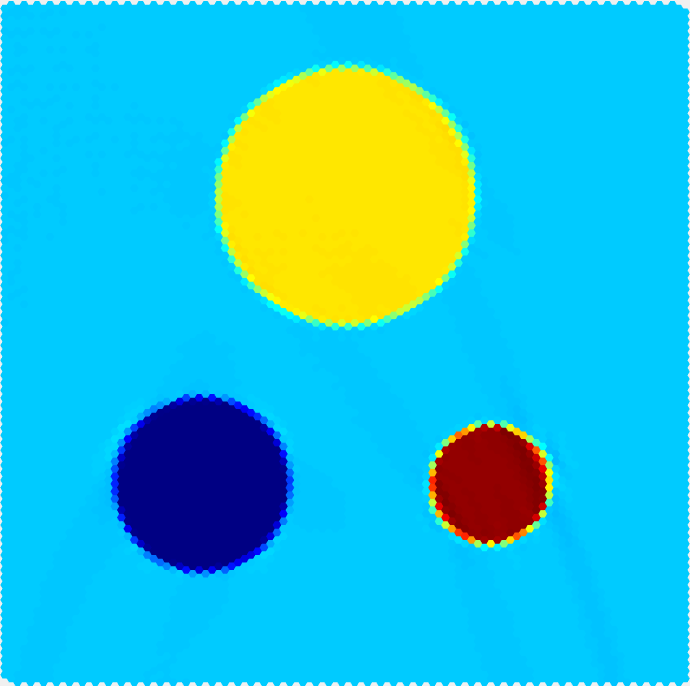};
\end{axis}
\end{tikzpicture}

 \begin{tikzpicture}[scale=\scale]
\begin{axis}[width=\witdh, height=\height, axis on top, scale only axis, xmin=\xmin, xmax=\xmax, ymin=\ymin, ymax=\ymax, colormap/jet, colorbar,point meta min=1,point meta max=4,ylabel=$(\PP^0{,}\PP^2)$, ylabel style={rotate=-90}]
\addplot graphics [xmin=\xmin,xmax=\xmax,ymin=\ymin,ymax=\ymax]{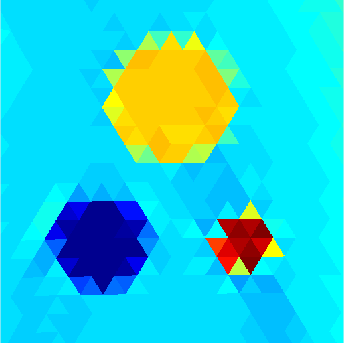};
\end{axis}
\end{tikzpicture}
 \begin{tikzpicture}[scale=\scale]
\begin{axis}[width=\witdh, height=\height, axis on top, scale only axis, xmin=\xmin, xmax=\xmax, ymin=\ymin, ymax=\ymax, colormap/jet, colorbar,point meta min=1,point meta max=4]
\addplot graphics [xmin=\xmin,xmax=\xmax,ymin=\ymin,ymax=\ymax]{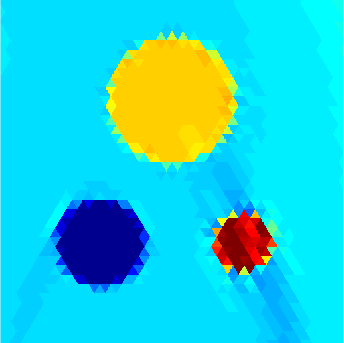};
\end{axis}
\end{tikzpicture}
 \begin{tikzpicture}[scale=\scale]
\begin{axis}[width=\witdh, height=\height, axis on top, scale only axis, xmin=\xmin, xmax=\xmax, ymin=\ymin, ymax=\ymax, colormap/jet, colorbar,point meta min=1,point meta max=4]
\addplot graphics [xmin=\xmin,xmax=\xmax,ymin=\ymin,ymax=\ymax]{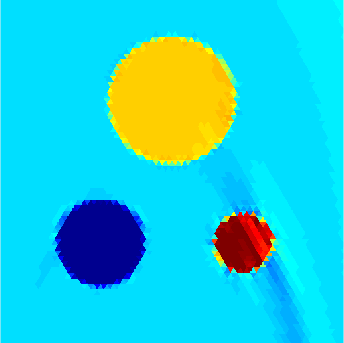};
\end{axis}
\end{tikzpicture}

\begin{tikzpicture}[scale=\scale]
\begin{axis}[width=\witdh, height=\height, axis on top, scale only axis, xmin=\xmin, xmax=\xmax, ymin=\ymin, ymax=\ymax, colormap/jet, colorbar,point meta min=1,point meta max=4,ylabel=$(\PP^1{,}\PP^2)$, ylabel style={rotate=-90}  ]
\addplot graphics [xmin=\xmin,xmax=\xmax,ymin=\ymin,ymax=\ymax]{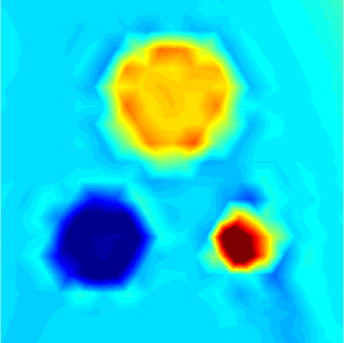};
\end{axis}
\end{tikzpicture}
 \begin{tikzpicture}[scale=\scale]
\begin{axis}[width=\witdh, height=\height, axis on top, scale only axis, xmin=\xmin, xmax=\xmax, ymin=\ymin, ymax=\ymax, colormap/jet, colorbar,point meta min=1,point meta max=4]
\addplot graphics [xmin=\xmin,xmax=\xmax,ymin=\ymin,ymax=\ymax]{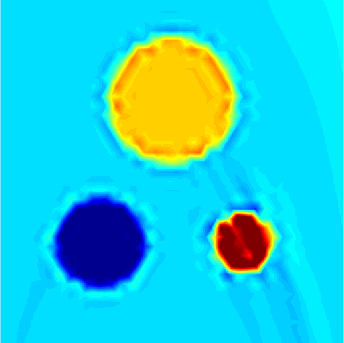};
\end{axis}
\end{tikzpicture}
 \begin{tikzpicture}[scale=\scale]
\begin{axis}[width=\witdh, height=\height, axis on top, scale only axis, xmin=\xmin, xmax=\xmax, ymin=\ymin, ymax=\ymax, colormap/jet, colorbar,point meta min=1,point meta max=4]
\addplot graphics [xmin=\xmin,xmax=\xmax,ymin=\ymin,ymax=\ymax]{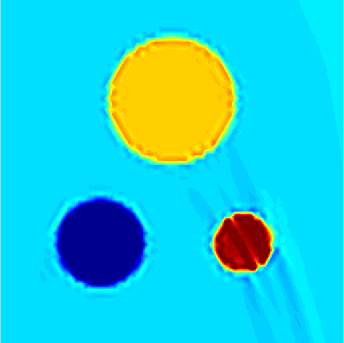};
\end{axis}
\end{tikzpicture}

 \begin{tikzpicture}[scale=\scale]
\begin{axis}[width=\witdh, height=\height, axis on top, scale only axis, xmin=\xmin, xmax=\xmax, ymin=\ymin, ymax=\ymax, colormap/jet, colorbar,point meta min=1,point meta max=4,ylabel=$(\PP^3{,}\PP^4)$, ylabel style={rotate=-90},title={$h=0.1$} ]
\addplot graphics [xmin=\xmin,xmax=\xmax,ymin=\ymin,ymax=\ymax]{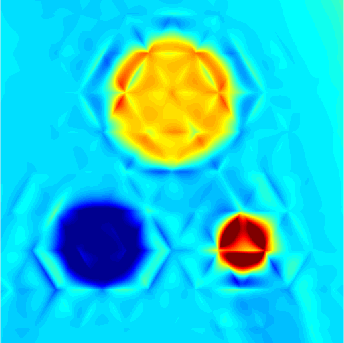};
\end{axis}
\end{tikzpicture}
 \begin{tikzpicture}[scale=\scale]
\begin{axis}[width=\witdh, height=\height, axis on top, scale only axis, xmin=\xmin, xmax=\xmax, ymin=\ymin, ymax=\ymax, colormap/jet, colorbar,point meta min=1,point meta max=4,,title={$h=0.05$} ]
\addplot graphics [xmin=\xmin,xmax=\xmax,ymin=\ymin,ymax=\ymax]{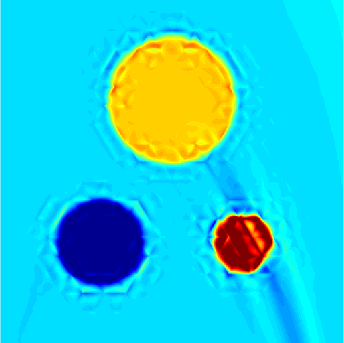};
\end{axis}
\end{tikzpicture}
 \begin{tikzpicture}[scale=\scale]
\begin{axis}[width=\witdh, height=\height, axis on top, scale only axis, xmin=\xmin, xmax=\xmax, ymin=\ymin, ymax=\ymax, colormap/jet, colorbar,point meta min=1,point meta max=4,title={$h=0.025$}]
\addplot graphics [xmin=\xmin,xmax=\xmax,ymin=\ymin,ymax=\ymax]{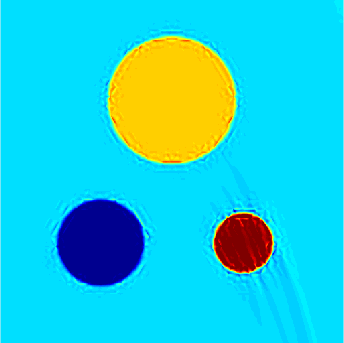};
\end{axis}
\end{tikzpicture}

\caption{\label{fig:resultelasto}Reconstruction of the shear modulus map $\mu$ using various pairs of finite element spaces in the 
subdomain of interest $(0.1,0.9)^2$.}
\end{center}\end{figure}


\begin{figure}\begin{center}
\small{
\begin{tabular}[title=a]{ |c||c|c|c|c| } 
\hline
$h=0.05$& $E$ & $n$ & $p$ \\
\hline
honeycomb & $9.2\%$ & 338 & 1888\\ 
$(\PP^0{,}\PP^2)$ &  $9.1\%$ & 735 & 2788 \\ 
$(\PP^1{,}\PP^2)$ & $8.5\%$ & 407 & 2800 \\ 
$(\PP^3{,}\PP^4)$ & $5.4\%$ & 3424 & 11k \\ 
 \hline
\end{tabular}
\hspace{1cm}
\begin{tabular}[title=a]{ |c||c|c|c| } 
\hline
$h=0.025$& $E$ & $n$ & $p$ \\
\hline
honeycomb & $6.3\%$ & 1510 & 8765 \\ 
$(\PP^0{,}\PP^2)$ & $6.7\%$ & 2982 & 12k \\ 
$(\PP^1{,}\PP^2)$ & $5.7\%$ & 1570 & 12k \\ 
$(\PP^3{,}\PP^4)$ & $3.4\%$ & 13654 & 47k \\ 
 \hline
\end{tabular}
}
\caption{\label{fig:table}Comparison of four pairs of finite element spaces in term of relative error 
$E$ of the reconstruction, degrees of freedom $n$, and number of equations $p$. The product $n\, p$ is an indication 
of the algorithmic complexity.}
\end{center}\end{figure}

\section{Concluding remarks}

In this article we have proved the numerical stability of the Galerkin 
approximation of the inverse parameter problem arising from the elastography in medical imaging. It as been done trough a direct discretization of the Reverse Weak Formulation without boundary conditions. The obtained stability estimates arises from a generalization of the \emph{inf-sup} constant (continuous and discrete) 
to a large class of first order differential operator. These results shed light on 
the importance of the choice of finite element spaces to assure uniqueness and stability. 
Various numerical applications have been presented that illustrate the stability theorems.
A new pair of finite element spaces based on an hexagonal tilling has been introduced. 
It showed excellent stability behavior for the specific purpose of this inverse problem.

\newpage

\appendix
\section{A result on self-adjoint operators}

\begin{lemma}\label{prop:app1} Let $H$ be an Hilbert space and $S:H\to H$ be a self-adjoint positive semi-definite linear operator. 
Call $\alpha^2:=\inf\{\left<Sx,x\right>_H|\ \norm{x}{H}=1\}$ and $z\in H$ such that $\norm{z}{H}=1$ and take $\left<Sz,z\right>_H\leq\alpha^2+\e^2$ with $\e> 0$. For any $p\perp z$ with $\norm{p}{H}=1$ we have
\begin{equation}\nonumber
\left|\inner{Sz,p}H\right|\leq \e\sqrt{\rho^2 -\alpha^2}
\end{equation}x
where $\rho^2:=\sup\{\left<Sx,x\right>_H|\ \norm{x}{H}=1\}$.
\end{lemma} 

\begin{proof} Consider $t\in (0,1)$, $u_t:=-\text{sign}\inner{Sz,p}H\sqrt{1-t^2}$ and $z_t:=t\, z+u_t\, p$ of norm one. By definition of $\alpha$ we have
\begin{equation}\nonumber
\begin{aligned}
\alpha^2 &\leq \inner{Sz_t,z_t}H=t^2\inner{Sz,z}H+ 2t\,u_t\inner{Sz,p}H+u_t^2\inner{Sp,p}H\\
 &\leq t^2(\alpha^2+\e^2)+ 2t\,u_t\inner{Sz,p}H+u_t^2\rho^2 .
\end{aligned}
\end{equation}
Then
\begin{equation}\nonumber
\begin{aligned}
 -2t\,u_t\inner{Sz,p}H &\leq (t^2-1)\alpha^2+t^2\e^2+u_t^2\rho^2 \\
  2t\,|u_t|\left|\inner{Sz,p}H\right| &\leq t^2\e^2+u_t^2(\rho^2 -\alpha^2)\\
    2\left|\inner{Sz,p}H\right| &\leq \frac t{|u_t|}\e^2+\frac {|u_t|}{t}(\rho^2 -\alpha^2).
\end{aligned}
\end{equation}
This statement is true for any $t\in(0,1)$ so for any $\tau\in (0,1)$ we have
\begin{equation}\nonumber
2\left|\inner{Sz,p}H\right| \leq\tau\e^2+\frac {1}{\tau}(\rho^2 -\alpha^2).
\end{equation}
The minimum of the right-hand side is reached for $\tau=\sqrt{(\rho^2 -\alpha^2)/\e^2}$ which implies that $2\left|\inner{Sz,p}H\right| \leq 2\sqrt{\e^2(\rho^2 -\alpha^2)}.$
\end{proof}

\section{Limit of subsets and infimum}

Let $M$ be a Hilbert space and let $E\subset M$ be Banach space dense in $M$. Let $(M_h)_{h>0}$ be a sequence of subspace of $E$ endowed with the $M$-norm. We assume that the orthogonal projection $\pi_h:M\to M_h$ satisfies
\begin{equation}\nonumber
\forall x\in E,\quad \norm{\pi_hx}{E}\leq \norm{x}{E}.
\end{equation}

\begin{definition}\label{def:setlimit} For any sequence $(A_h)_{h>0}$ of subsets of $M$, we define its limit as 
\begin{equation}\nonumber
\lim_{h\to 0}A_h:=\left\{x\in M|\, \exists (x_h)_{h>0}\subset M,\ \lim_{h\to 0}\norm{x_h-x}{M}=0,\  \forall h>0\ x_h\in A_h\right\}.
\end{equation}
\end{definition}

\begin{proposition} $\lim_{h\to 0}A_h$ is a closed subset of $M$ and, if $A_h\subset X\subset M$ for all $h>0$, then $\lim_{h\to 0}A_h\subset \overline{X}$.
\end{proposition}

\begin{proof}  Call $A:=\lim_{h\to 0}A_h$ and take $x\in \overline{A}$. There exists a sequence $(x_n)_{n\in \N}$ of $A$ such that $\norm{x-x_n}{M}\leq 1/(2n)$ for all $n\in\N^*$. For all $n\in\N^*$, there exists a sequence $(x_n^h)_{h>0}$ such that $\lim_{h\to 0}\norm{x^h_n-x_n}{M}=0$ and $x_n^h\in A_h$ for all $h>0$. Hence there exists $h_n>0$ such that for all $h\leq h_n$ we have  $\norm{x_n^h-x_n}{M}\leq 1/(2n)$. We can decrease $h_n$ to satisfy $h_n < h_{n-1}$ for all $n\geq 2$. Now define the sequence $(y_h)_{h>0}$ as follows: If $h>h_1$, $y_h$ is any element of $A_h$. If $h\in [h_{n+1},h_n)$, we take $y_h = x_n^{h}$. It is clear that $y_h\in A_h$ for all $h>0$. Moreover, for any $h\leq h_n$, $\norm{y_h-x_n}{M}\leq 1/(2n)$ and $\norm{x-x_n}{M}\leq 1/(2n)$ which give $\norm{y_h-x}{M}\leq 1/n$. This shows that $\lim_{h\to 0}\norm{y_h-x}{M}=0$ and therefore $x\in A$. The second part of the statement is trivial. 
 
\end{proof}

\begin{proposition}\label{prop:limsup} Assume that $A:=\lim_{h \to 0} A_h$ is not empty and consider a fonction $f:M\to \R$. If there exists a subset $B\subset A$ such that $f$ is continuous in $B$ and $\inf_Af=\inf_Bf$ then we have 
\begin{equation}\nonumber
\limsup_{h\to 0}\inf_{A_h}f\leq \inf_Af.
\end{equation}

\end{proposition}

\begin{proof} Take $x\in B$. As $x\in A$, there exists $(x_h)_{h>0}$ such that $x_h\in A_h$ for all $h>0$ and  $\lim_{h\to 0}x_h=x$. For any $h>0$,
 $f(x_h) \leq f(x)+|f(x_h)-f(x)|$ and $\inf_{A_h} f \leq f(x)+|f(x_h)-f(x)|$.
Taking the superior limit when $h\to 0$ it comes from the continuity of $f$ at $x$, $\limsup_{h\to 0}\inf_{A_h}f\leq f(x)$ 
which if true for any $x\in B$ so $\displaystyle{\limsup_{h\to 0}\inf_{A_h}f\leq \inf_B f = \inf_A f}$.
\end{proof}
We assume now that the sequence $(M_h)$ satisfies $\lim_{h\to 0} M_h = M$. We consider a sequence of positive real number $\alpha_h$ that converges zero and a corresponding sequence of subsets $C_h:=\left\{x\in M_h\, |\, \alpha_h\norm{x}{E}\leq \norm{x}{M}\right\}$.

\begin{proposition}\label{prop:lim1} The following limit holds: $\displaystyle \lim_{h\to 0}C_h=M.$
\end{proposition}

\begin{proof}  We prove that $E\subset C:=\lim_{h\to 0}C_h$. Take $x\in E\bs \{0\}$, for $h$ small enough it satisfies  $2\alpha_h\norm{x}{E}\leq \norm{x}{M}$. Consider now its orthogonal projection $\pi_hx$ of $x$ onto $M_h$. It satisfies $\lim_{h\to 0}\pi_hx=x$. For $h$ small enough $\norm{x}{M}\leq 2\norm{\pi_hx}{M}$ and then
\begin{equation}\nonumber
\alpha_h\norm{\pi_hx}{E}\leq \alpha_h\norm{x}{E}\leq \frac{1}{2}\norm{x}{M}\leq \norm{\pi_hx}{M}
\end{equation}
which means that $\pi_hx\in C_h$. As a consequence, $x\in \lim_{h\to 0}C_h$.
\end{proof}

\begin{proposition}\label{prop:wc} Let $(z_h)_{h>0}$ be sequence of $M$ such that $\norm{z_h}{M}=1$ and which converges weakly to $z\neq 0$. Then 
\begin{equation}\nonumber
\lim_{h\to 0}\left(C_h\cap \{z_h\}^\perp\right)=M\cap \{z\}^\perp.
\end{equation}
\end{proposition}

\begin{proof} Take $x\in \lim_{h\to 0}\left(C_h\cap \{z_h\}^\perp\right)$. There exists $(x_h)$ such that $x_h\in C_h$ and $x_h\perp z_h$ and $x_h\to x$. We have $\inner{x,z}{M}=\lim_{h\to 0}\inner{x,z_h}{M}=\lim_{h\to 0}\inner{x-x_h,z_h}{M}=0$. Then $x\in M\cap \{z\}^\perp$.

Reversely, take $x\in M\cap \{z\}^\perp$, and fix $\e>0$. There exists $x_\e\in E\bs\{0\}$ such that $\norm{x_\e-x}{M}\leq \e$ and $x_\e\perp z$ and Consider now the orthogonal projection $\pi_hx_\e$ of $x_\e$ onto $M_h$. It satisfies $\lim_{h\to 0}\pi_hx_\e=x_\e$. For $h$ small enough $\norm{x_\e}{M}\leq 2\norm{\pi_hx_\e}{M}$. Consider now $\wt z\in E$ such that $\inner{z,\wt z}{M}\geq 1/2$ and $\norm{\wt z}{M}=1$. We define now
\begin{equation}\nonumber
x^h_\e=\pi_hx_\e+\beta_h \pi_h\wt z\quad \in M_h,
\end{equation}

with $\beta_h = -\inner{\pi_hx_\e,z_h}{M}/\inner{\pi_h\wt z,z_h}{M}$ in order to have $x^h_\e\perp z_h$ for all $h$. Remark that $\beta_h$ is wel defined for $h$ small enough as $\inner{\pi_h\wt z,z_h}{M}$ converges to $\inner{z,\wt z}{M}$ and converges to zero as $\inner{\pi_hx_\e,z_h}{M}=\inner{x_\e,z_h}{M}+\inner{\pi_hx_\e-x_\e,z_h}{M}$ converges to $\inner{x_\e,z}{M}=0$. Then $x^h_\e\to x_\e$. Now we write
\begin{equation}\nonumber\begin{aligned}
\norm{x^h_\e}E &\leq\norm{\pi_hx_\e}E+\beta_h \norm{\pi_h\wt z}E\leq \norm{x_\e}E+\beta_h \norm{\wt z}E,\\
\end{aligned}\end{equation}
and $\norm{x^h_\e}M\to \norm{x_\e}{M}\neq 0$. As a consequence, for $h$ small enough, $\alpha_h\norm{x^h_\e}E\leq \norm{x^h_\e}M$ which means that $x^h_\e\in C_h\cap \{z_h\}^\perp$  for $h$ small enough. This shows that $x_\e\in \lim_{h\to 0}\left(C_h\cap \{z_h\}^\perp\right)$. This is true for any $\e>0$ and as the limit set is closed, $x\in \lim_{h\to 0}\left(C_h\cap \{z\}^\perp\right)$.
\end{proof}

\section*{Acknowledgments}
The authors acknowledge support from the LABEX MILYON (ANR-10-LABX-0070) of Universit\'e de Lyon, within the program "Investissements d'Avenir" (ANR-11-IDEX- 0007) operated by the French National Research Agency (ANR). 

\bibliographystyle{plain}
\bibliography{biblio}
\end{document}